\documentclass[a4paper,final]{article}
\usepackage[british]{babel}
\usepackage{amsmath, amssymb, amsthm, amsfonts}
\usepackage{ifdraft}
\usepackage{microtype}
\usepackage[svgnames]{xcolor}
\usepackage{url}
\usepackage[utf8]{inputenc}
\usepackage{aliascnt}
\usepackage[colorlinks=true,linkcolor=blue,draft=false]{hyperref}
\usepackage{tikz}
\usepackage{mathtools}
\usepackage{enumerate}
\usepackage{extarrows}
\usepackage{enumitem}

\usepackage[notref,notcite,color]{showkeys}

\usepackage{algorithmic}

\ifdraft{
\newcommand{\colorS}[1]{{\color{blue} #1}}
\newcommand{\colorV}[1]{{\color{red} #1}}
\newcommand{\itemS}[1]{\colorS{\item[S:] #1}}
\newcommand{\itemV}[1]{\colorV{\item[V:] #1}}
\newcommand{\commS}[1]{{\marginpar{\tiny \colorS{#1}}}}
\newcommand{\commV}[1]{{\marginpar{\tiny \colorV{#1}}}}
}
{
\newcommand{\colorS}[1]{}
\newcommand{\colorV}[1]{}
\newcommand{\itemS}[1]{}
\newcommand{\itemV}[1]{}
\newcommand{\commS}[1]{}
\newcommand{\commV}[1]{}
}

\makeatletter
\def\myfooter{\xdef\@thefnmark{}\@footnotetext}
\makeatother

\title{Zappa--Sz\'ep products of Garside monoids}
\author{Volker Gebhardt, Stephen Tawn}
\date{24th October 2015}

\hypersetup{pdftitle={Zappa-Sz\'ep products of Garside monoids},
            pdfauthor={Volker Gebhardt, Stephen Tawn},
            pdfkeywords={},
            pdfsubject={},
            pdfdisplaydoctitle={true}}

\newcommand{\NN}{\mathbb{N}}

\newcommand{\from}{\colon\!}

\newcommand{\Union}{\bigcup}
\newcommand{\intersect}{\cap}

\def\setminus{\mathbin{\raisebox{0.25ex}{$\smallsetminus$}}}
\newcommand{\join}{\vee}
\renewcommand{\Join}{\bigvee}
\newcommand{\meet}{\wedge}
\newcommand{\Meet}{\bigwedge}

\newcommand{\leftJoin}{\Join}
\newcommand{\rightjoin}{\mathbin{\widetilde{\join}}}
\newcommand{\rightJoin}{\mathop{\widetilde{\Join}}}

\newcommand{\rightmeet}{\mathbin{\widetilde{\meet}}}

\let\epsilon\varepsilon
\newcommand{\prefix}{\preccurlyeq}
\newcommand{\suffix}{\succcurlyeq}
\newcommand{\factor}{\dot\curlyvee}
\newcommand{\under}{\backslash}
\newcommand{\ldot}{\mathbin{.}}

\newcommand{\rightacts}{\vartriangleleft}
\newcommand{\leftacts}{\vartriangleright}
\newcommand{\Rightacts}{\blacktriangleleft}
\newcommand{\Leftacts}{\blacktriangleright}

\DeclareMathOperator{\Div}{Div}

\newcommand{\Atoms}{\mathcal{A}}
\newcommand{\Simples}{\mathcal{D}}
\newcommand{\pSimples}{\Simples^{\!{}^{\circ}\!}}
\newcommand{\cl}{\mathrm{cl}}
\renewcommand{\complement}{\partial}
\newcommand{\rightcomplement}{\widetilde{\complement}}
\newcommand{\id}{\mathbf{1}}
\newcommand{\op}{\mathrm{op}}
\newcommand{\QZ}{\mathrm{QZ}}

\newcommand{\NF}{\mathrm{NF}}
\newcommand{\ZS}{Zappa--Sz{\'e}p}

\let\zs=\bowtie

\newcommand{\calG}{\mathcal{G}}
\newcommand{\calH}{\mathcal{H}}
\newcommand{\calL}{\mathcal{L}}

\newcommand{\calLbar}{\overline{\mathcal{L}}}

\theoremstyle{plain}
\newtheorem{theorem}{Theorem}

\newaliascnt{lemma}{theorem}
\newtheorem{lemma}[lemma]{Lemma}
\aliascntresetthe{lemma}

\newaliascnt{proposition}{theorem}
\newtheorem{proposition}[proposition]{Proposition}
\aliascntresetthe{proposition}

\newaliascnt{corollary}{theorem}
\newtheorem{corollary}[corollary]{Corollary}
\aliascntresetthe{corollary}

\theoremstyle{remark}

\newtheorem*{claim*}{Claim}
\newtheorem*{remark}{Remark}

\theoremstyle{definition}

\newaliascnt{definition}{theorem}
\newtheorem{definition}[definition]{Definition}
\aliascntresetthe{definition}

\newaliascnt{example}{theorem}
\newtheorem{example}[example]{Example}
\aliascntresetthe{example}

\newaliascnt{notation}{theorem}
\newtheorem{notation}[notation]{Notation}
\aliascntresetthe{notation}

\addto\extrasbritish{

}

\makeatletter
\def\MR@url=#1 #2={http://www.ams.org/mathscinet-getitem?mr=#1}

\makeatother

\begin{document}

\maketitle
\myfooter{Both authors acknowledge support under Australian Research Council's Discovery Projects funding scheme (project number DP1094072). Volker Gebhardt acknowledges support under the Spanish Projects MTM2010-19355 and MTM2013-44233-P.}
\myfooter{MSC-class: 20F36 (Primary) 20M13, 06F05 (Secondary)}

\begin{abstract}
A monoid~$K$ is the internal Zappa--Sz{\'e}p product of two submonoids, if every element of~$K$ admits a unique factorisation as the product of one element of each of the submonoids in a given order.  This definition yields actions of the submonoids on each other, which we show to be structure preserving.

We prove that~$K$ is a Garside monoid if and only if both of the submonoids are Garside monoids.  In this case, these factors are parabolic submonoids of~$K$ and the Garside structure of~$K$ can be described in terms of the Garside structures of the factors.
We give explicit isomorphisms between the lattice structures of~$K$ and the product of the lattice structures on the factors that respect the Garside normal forms.  In particular, we obtain explicit natural bijections between the normal form language of~$K$ and the product of the normal form languages of its factors.
\end{abstract}

\section{Introduction}\label{S:Introduction}

The notion of \ZS{} products generalises those of direct and semidirect products; the key property is that every element of the \ZS{} product can be written uniquely as a product of two elements, one from each factor, in any given order.

For instance, a group $K$ is the (internal) \emph{\ZS{} product} of two subgroups~$G$ and~$H$, written
$K=G\zs H$, if for every $k\in K$ there exist unique elements $g\in G$ and $h\in H$ such that $g h = k$, or equivalently, if $K=GH$ and $G\cap H=\{\id\}$.  As taking inverses is an anti-isomorphism, one also has $K=HG$ and one obtains unique elements $g'\in G$ and $h'\in H$ such that $h' g' = k$.  However, in general neither $g=g'$ nor $h=h'$ need to hold.
The special case that one of the factors, say $G$, is a normal subgroup yields a semidirect product $G\rtimes H$; if both factors are normal, one obtains the direct product $G\times H$.

Note that if we consider the case of monoids, the symmetry under swapping the factors is not automatic, that is, $K=GH$ need not imply $K=HG$.
\medskip

\noindent
\ZS{} products have been studied for various categories of algebraic objects by many authors; see for instance \cite{Casadio41,Zappa42,Szep50,Szep51,RedeiSzep55,Szep62,Takeuchi81,Kunze83,PicantinPhD,Picantin01,Brin05,Brin07,Godelle10,Agore09,AgoreMilitaru11,Brownlowe}.
There are subtle differences in definitions between some of these references, for instance regarding the symmetry under swapping of factors.
\medskip

\noindent
In the context of Garside monoids, \ZS{} products were studied by Picantin \cite{Picantin01}; he used the term \emph{crossed products}.
Given a family $M_1,\ldots,M_\ell$ of Garside monoids and a family of maps $\Theta_{i,j}:M_i\times M_j\to M_j$ that satisfy certain compatibility conditions, Picantin constructs a Garside structure on the set $M_1\times\cdots\times M_\ell$, that is, he considers \emph{external} products.  (He also uses this construction for decomposing Garside monoids, namely to show that every Garside monoid is the iterated crossed product of Garside monoids that have a cyclic centre.)

The construction in~\cite{Picantin01} is, however, very technical. The difficulty comes, firstly, from the compatibility conditions for the maps~$\Theta_{i,j}$, which are needed to ensure that the way in which the factors~$M_i$ are made to interact is consistent and the crossed product is well-defined;
and, secondly, from the fact that crossed products on more than two factors are considered.
In the setting of~\cite{Picantin01}, both considering external products (which require explicit compatibility conditions) and considering products involving more than two factors is necessary, so these complications are unavoidable.

Our situation is different:  We are primarily interested in \emph{decomposing} a given Garside monoid into simpler components that are also Garside monoids.  In this situation, the way in which the potential factors interact is defined by the ambient monoid and there is no need for explicit compatibility conditions.  It is therefore natural for us to consider \emph{internal} \ZS{} products:

\begin{definition}\label{D:ZappaSzepProduct}
  \rightskip30mm
  Given a monoid $K$ with two submonoids~$G$ and~$H$, say that~$K$ is
  the (internal) \emph{\ZS{} product} of~$G$ and~$H$, written
  $K=G\zs H$, if for every $k\in K$ there exist unique $g_1, g_2\in
  G$ and $h_1, h_2 \in H$ such that $g_1 h_1 = k = h_2 g_2$.

  We will say that $g_1 h_1$ is the $G H$-decomposition of $k$ and
  that $h_2 g_2$ is its $H G$-decomposition.
  \hfill
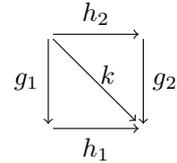
\begin{picture}(0,0)(-10,0)
  \begin{tikzpicture}[scale=1.25]
    \draw [->] (0.05,0.95) -- (0.92,0.08) node [right] at (0.45,0.57) {$k$};
    \draw [->] (0,0.95) -- (0,0.05) node [left] at (0,0.5) {$g_1$};
    \draw [->] (0.05,0) -- (0.95,0) node [below] at (0.5,0) {$h_1$};
    \draw [->] (0.05,1) -- (0.95,1) node [above] at (0.5,1) {$h_2$};
    \draw [->] (1,0.95) -- (1,0.05) node [right] at (1,0.5) {$g_2$};
  \end{tikzpicture}
  \end{picture}
  \par
  \rightskip0mm
\end{definition}
It is obvious from this definition that forming internal \ZS{} products is commutative (that is, $K=G\zs H$ if and only if $K=H\zs G$).
It is, however, not associative, that is $K = (F \zs G) \zs H$, meaning that there exists a submonoid $K'$ such that $K = K' \zs H$ and $K' = F \zs G$, does not imply $K = F \zs (G\zs H)$; see \autoref{E:SemidirectProduct}.
As our construction can easily be applied iteratively, we can restrict to the case of two factors.
\medskip

\noindent
Picantin shows that, for each
ordering of the factors in a crossed product, an element of the
product can be written uniquely as a product of elements of the
factors in that order~\cite[Proposition~3.6]{Picantin01}.
Thus, every crossed product is also an internal \ZS{} product in the sense of \autoref{D:ZappaSzepProduct}, so in particular, every monoid that can be decomposed as a crossed product can also be decomposed as an internal \ZS{} product.
\emph{A priori}, decomposability as an internal \ZS{} product could be weaker (cf.\ \autoref{E:NotInjective}), although it turns out that the notions are equivalent in the case of Garside monoids; cf.\ \autoref{C:ZS-crossed}.%
\medskip

In the context of general (i.e., non Garside) modoids and categories, both internal and external \ZS{} products are discussed in detail in~\cite{Brin05}; see also~\cite{Brin07,Brownlowe}.
\cite{Brin05} defines both internal and external \ZS{} products of monoids, and proves that both notions are equivalent~\cite[Lemma~3.9]{Brin05}.  (The author calls the external version ``not attractive''.)

It must be emphasised, however, that external \ZS{} products in the sense of~\cite{Brin05} are not the same as crossed products in the sense of~\cite{Picantin01}:
\autoref{E:NotInjective} gives a monoid that is the (internal and external) \ZS{} product of two submonoids, but that cannot be decomposed as a crossed product.
In particular, it cannot be concluded from~\cite[Lemma~3.9]{Brin05} that the notions of internal \ZS{} products (\autoref{D:ZappaSzepProduct}) and crossed products coincide.
It follows with the results from~\cite{Picantin01} and~\autoref{action-bijections} that they do coincide in the situation of Garside monoids (cf.\ \autoref{C:ZS-crossed}), but this is not obvious \emph{a priori}.
\bigskip

\noindent
The main contributions of this paper to the concept of \ZS{} products of Garside monoids which was introduced in~\cite{Picantin01} are the following:\vspace{-\itemsep}
\begin{enumerate}[itemindent=6mm,listparindent=\parindent,itemsep=\parsep,leftmargin=0cm]
\item
We develop a notion of internal \ZS{} products of Garside monoids which is less technical than that of crossed products and more suitable for analysing decompositions of Garside monoids:
Our definition of an internal \ZS{} product assumes \emph{only} the existence of unique decompositions of elements;
from this definition we derive natural maps which we prove to be bijective actions preserving algebraic structures in certain situations.
Our approach in this regard is analogous to the one in~\cite{Brin05}.
The definition of crossed products in~\cite{Picantin01} features similar maps, but it imposes various properties of the latter as requirements.

While decomposability as a crossed product and decomposability as an internal \ZS{} product turn out to be equivalent in the case of Garside monoids, this is not at all obvious from the definitions.
Indeed, this equivalence suggests that the definitions of crossed products and internal \ZS{} products are in some sense optimal, as they describe a structure that is obtained naturally by both, constructing and deconstructing Garside monoids as products of submonoids.

\item Our techniques allow us to relate the regular language of normal forms in a \ZS{} product to the languages of normal forms in its factors; there are no similar results in~\cite{Picantin01}:
We give explicit translations, in both directions, between an automaton accepting the regular language of normal forms in a \ZS{} product on the one hand and a pair of automata accepting the regular languages of normal forms in the factors on the other hand.  We also give explicit algorithms transforming normal forms.
\end{enumerate}
\medskip

\noindent
The structure of the paper is as follows:
In \autoref{S:Background} we recall the main concepts used in the paper and fix notation.
In \autoref{S:Actions} we define actions of the factors of a \ZS{} product on each other and analyse their properties.
\autoref{S:ActionsForGarsideMonoids} is devoted to the case that either the \ZS{} product of two monoids is a Garside monoid or that both of the factors are; we will show that both conditions are equivalent.
Finally, in \autoref{S:Garside-ZS} we consider the situation where the Garside elements of the factors and of the product are chosen in a compatible way; we show that in this case the regular language of normal forms in the product can be described effectively in terms of those of the factors.

\section{Background}\label{S:Background}

In order to fix notation, we briefly recall the main concepts used in the paper.  For details we refer to \cite{DehornoyParis99,Dehornoy02,GarsideBook}.

Let $M$ be a monoid and $\id$ its identity.
The monoid $M$ is called \emph{left-cancellative} if for any $x,y,y'$ in $M$, the equality $xy = xy'$ implies $y = y'$.
Similarly, $M$ is called \emph{right-cancellative} if for any $x,y,y'$ in $M$, the equality $yx = y'x$ implies $y = y'$.

For $x,y\in M$, we say that $x$ is a \emph{left-divisor} or \emph{prefix} of $y$, writing $x \prefix_M y$, if there exists an element $u\in M$ such that $y = xu$.  If the monoid is obvious, we simply write $x \prefix y$ to reduce clutter.
Similarly, we say that $x$ is a \emph{right-divisor} or \emph{suffix} of $y$,
writing $y\suffix_M x$ or $y\suffix x$, if there exists $u\in M$ such that $y = ux$.
Moreover, we say that $x$ is a \emph{factor} of $y$, writing $x\factor y$, if there exist elements $u,v\in M$ such that $y=uxv$.
If $M$ does not contain any non-trivial invertible elements, then the relation $\prefix$ is a partial order if $M$ is left-cancellative, and the relation $\suffix$ is a partial order if $M$ is right-cancellative.

An element $a\in M\setminus\{\id\}$ is called an \emph{atom} if whenever $a = uv$ for $u,v\in M$, either $u = \id$ or $v = \id$ holds.  The existence of atoms implies that $M$ does not contain any non-trivial invertible elements.
The monoid $M$ is said to be
\emph{atomic} if it is generated by its set $\Atoms$ of atoms and if for every
element $x\in M$ there is an upper bound on the length of
decompositions of $x$ as a product of atoms, that is, if
$||x||_{\Atoms} := \sup\{ k\in\NN : x=a_1\cdots a_k \text{ with }a_1,\ldots,a_k\in\Atoms \} < \infty$.

An element $d\in M$ is called \emph{balanced}, if the set of its left-divisors is equal to the set of its right-divisors.  In this case, we write $\Div(d)$ for the set of (left- and right-) divisors of $d$.
\begin{definition}
  A \emph{quasi-Garside structure} is a pair $(M, \Delta)$ where $M$
  is a monoid and $\Delta$ is an element of $M$ such that
  \begin{enumerate}[label=(\alph*)] \itemsep 0em
  \item $M$ is cancellative and atomic,
  \item the prefix and suffix relations are lattice orders, that is, for any pair of elements there exist unique least common upper bounds and unique greatest common lower bounds with respect to $\prefix$ respectively $\suffix$,
  \item $\Delta$ is balanced, and
  \item $M$ is generated by the divisors of $\Delta$.
  \end{enumerate}
  If the set of divisors of $\Delta$ is finite then we say that $(M,
  \Delta)$ is a \emph{Garside structure}.

  A monoid $M$ is a (quasi)-Garside monoid if there exists a
  \emph{(quasi)-Garside element} $\Delta \in M$ such that $(M, \Delta)$ is a
  (quasi)-Garside structure.
\end{definition}
\begin{remark}
  If $M$ is a Garside monoid then the choice of Garside element is not
  unique.  Indeed, if $\Delta$ is a Garside element then $\Delta^\ell$ is
  also a Garside element for all $\ell \in \NN$.

  If $(M,\Delta)$ is a quasi-Garside structure in the above sense, then in the terminology of \cite{GarsideBook}, the set $\Div(\Delta)$ forms a bounded Garside family for the monoid $M$.
  The elements of $\Div(\Delta)$ are called \emph{simple elements}.  (Note that the set of simple elements depends on the choice of the Garside element.)
\end{remark}

\begin{definition}
  A monoid $M$ is \emph{conical} if for all $x,y \in M$, $x y = \id$
  implies that $x=\id=y$.  In particular, all Garside monoids are
  conical.
\end{definition}

\begin{notation}
If $M$ is a left-cancellative atomic monoid, then least common upper bounds and greatest common lower bounds are unique if they exist.  In this situation, we will write $x \join y$ for the $\prefix$-least common upper bound of $x,y\in M$ if it exists, and we write $x \meet y$ for their $\prefix$-greatest common lower bound if it exists.
If $x,y\in M$ admit a $\prefix$-least common upper bound, we define $x\under y$ as the unique element of $M$ satisfying $x(x\under y)=x\join y$.

Similarly, if $M$ is a right-cancellative atomic monoid, we will write $x \rightjoin y$ and $x \rightmeet y$ for the $\suffix$-least common upper bound respectively the $\suffix$-greatest common lower bound of $x$ and $y$ if they exist, and if $x$ and $y$ admit a $\suffix$-least common upper bound, we define $y/x$ as the unique element of $M$ satisfying $(y/x)x=x\rightjoin y$.

\begin{lemma}[{\cite[Lemme 1.7]{Dehornoy02}}]\label{L:Complements}
If $M$ is a Garside monoid, then one has
\begin{minipage}{7cm}%
\begin{align*}
   (ab)\under c &= b\under (a\under c) \\
   c\under(ab) &= (c\under a)\cdot \big((a\under c)\under b\big)
\end{align*}
for any $a,b,c\in M$.
\end{minipage}
\hfill
\begin{minipage}{4.25cm}%
\begin{tikzpicture}[scale=1.25]
  \draw [->] (0,0.95) -- (0,0.05) node [left] at (0,0.5) {$c$};
  \draw [->] (0.05,0) -- (0.95,0) node [below] at (0.5,0) {$c\under a$};
  \draw [->] (1.05,0) -- (1.95,0) node [below] at (1.5,0) {$(a\under c)\under b$};
  \draw [->] (0.05,1) -- (0.95,1) node [above] at (0.5,1) {$a$};
  \draw [->] (1.05,1) -- (1.95,1) node [above] at (1.5,1) {$b$};
  \draw [->] (1,0.95) -- (1,0.05) node [right] at (1,0.5) {$a\under c$};
  \draw [->] (2,0.95) -- (2,0.05) node [right] at (2,0.5) {$b\under(a\under c)$};
  \end{tikzpicture}
\end{minipage}
\end{lemma}%

\noindent
If $(M,\Delta)$ is a Garside structure, we write $\Simples_M$ for the set of simple elements $\Div(\Delta)$, and we define the set of \emph{proper} simple elements as $\pSimples_M = \Simples_M \setminus \{ \id, \Delta \}$, where $\id$ is the identity element of~$M$.  To avoid clutter, we will usually drop the subscript if there is no danger of confusion.
For $x\in\Simples$, there exists a unique element $\complement x=\complement_M x\in\Simples$ such that $x \complement x=\Delta$.  Clearly, $\complement x\in\pSimples$ iff $x\in\pSimples$.
Moreover, for $x\in M$, we define $\Delta_x :=  \Delta_x^{M} := \Join\{ y\under x : y \in M \}$%
~\cite{Picantin01}.

Given a set $X$ we will write $X^* = \Union_{i = 0}^{\infty} X^i$ for
the set of strings of elements of $X$.  We will write $\epsilon$ for
the empty string and separate the letters of a string with dots, for
example we will write $a \ldot b \ldot a \in \{a,b\}^*$.

Given a (quasi)-Garside structure $(M, \Delta)$ we can define the \emph{left
normal form} of an element by repeatedly extracting the
$\prefix$-GCD of the element and $\Delta$.  More precisely, the
normal form of $x \in M$ is the unique word $\NF(x) = x_1 \ldot x_2
\ldot \cdots \ldot x_\ell$ in $(\Simples\setminus\{\id\})^*$ such that $x = x_1 x_2 \cdots
x_\ell$ and $x_i = \Delta \meet x_i x_{i+1} \cdots x_\ell$ for $i=1,\ldots,\ell$, or equivalently, $\complement x_{i-1} \meet x_i=\id$ for $i=2,\ldots,\ell$.
We write $x_1|x_2|\cdots|x_\ell$ for the word $x_1 \ldot x_2 \ldot \cdots \ldot x_\ell$
together with the proposition that this word is in normal form.

If $x_1|x_2|\cdots|x_\ell$ is the normal form of $x\in M$, we define the \emph{infimum} of $x$ as
$\inf(x) = \max\{ i\in\{1,\ldots,\ell\} : x_i = \Delta \}$, the \emph{supremum} of $x$ as $\sup(x) = \ell$, and the \emph{canonical length} of $x$ as $\cl(x) = \sup(x)-\inf(x)$.  Note that $\inf(x)$ is the largest integer~$i$ such that $\Delta^i\prefix x$ holds, and $\sup(x)$ is the smallest integer~$i$ such that $x\prefix \Delta^i$ holds.

Let $\calL$ be the language on the set $\pSimples$ of proper simple elements
consisting of all words in normal form, and write $\calL^{(n)}$ for
the subset consisting of words of length $n$:
\[
  \calL := \bigcup_{n\in\NN} \calL^{(n)} \quad\text{where}\quad
  \calL^{(n)} := \big\{ 
      x_1 \ldot \cdots \ldot x_n \in \big(\pSimples\big)^* :
      \forall i,\ \complement x_i \meet x_{i+1} = \id 
    \big\}
\]

We also define
\[
 \calLbar := \bigcup_{n\in\NN} \calLbar^{(n)} \quad\text{where}\quad
  \calLbar^{(n)} := \big\{ 
      x_1 \ldot \cdots \ldot x_n \in (\Simples\setminus\{\id\})^* :
      \forall i,\ \complement x_i \meet x_{i+1} = \id 
    \big\} \;.
\]
\end{notation}

\begin{definition}[{\cite[Definition 2.2]{Godelle07}}]\label{D:ParabolicSubmonoid}
Let $M$ be a Garside monoid with set of atoms $\Atoms$,
let $\delta$ be a balanced simple element of $M$, and let $M_\delta$ be the submonoid of~$M$ generated by
$\{ a\in \Atoms : a \prefix\delta\}$.

$M_\delta$ is a \emph{parabolic submonoid} of $M$, if
$\{ x \in M : x\prefix \delta \} = \Simples \cap M_\delta$ holds.
\end{definition}

\begin{proposition}[{\cite[Lemma 2.1]{Godelle07}}]\label{P:ParabolicSubmonoid}
If $M$ is a Garside monoid and~$\delta$ is a balanced simple element of~$M$ such that~$M_\delta$ is a parabolic submonoid of~$M$, then $M_\delta$ is a sublattice of $M$ for both $\prefix$ and $\suffix$ that is closed under the operations $\under$ and $/$.
In particular, $M_\delta$ is a Garside monoid with Garside element $\delta$.
\end{proposition}

\begin{remark}
If $M_\delta$ is a parabolic submonoid of~$M$, then for any $x\in M_\delta$, the left normal form of~$x$ in the Garside monoid $M_\delta$ coincides with its left normal form of~$x$ in the Garside monoid~$M$.
\end{remark}
\medskip

\begin{definition}[\cite{Picantin01}]\label{D:QuasiCentre}
For a Garside monoid $M$ with set of atoms $\Atoms$ we define the \emph{quasi-centre} of~$M$ as
$\QZ := \QZ_M := \{ u\in M : \Atoms u = u \Atoms \}$, and we say that~$M$ is \emph{$\Delta$-pure} if $\Delta_a=\Delta_b$ holds for any $a,b\in\Atoms$.
\end{definition}

\begin{proposition}[\cite{Picantin01}]\label{P:QuasiCentre}
Let $M$ be a Garside monoid with set of atoms $\Atoms$.
 \begin{enumerate}[label=\upshape{(\alph*)}]
  \item For any $x\in M$ and $c\in\QZ$, one has $x\prefix c \Longleftrightarrow c\suffix x \Longleftrightarrow x\factor c$. 
  \item For any $a,b\in \Atoms$, one has either $\Delta_a = \Delta_b$ or $\Delta_a\meet\Delta_b=\id$. 
  \item If $c_1, c_2\in \QZ$, then $c_1\meet c_2\in \QZ$. 
  \item For any $x\in M$, one has $\Delta_x = \Meet ( \QZ \intersect xM )$. 
    In particular, $\Delta_x\in\QZ$ and $x \prefix \Delta_x$.
  \item For any $x,y\in M$, one has $\Delta_{x\join y} = \Delta_x\join\Delta_y$. 
  \item $\QZ$ is a free abelian monoid with basis $\{ \Delta_a : a\in\Atoms \}$. 
 \end{enumerate}
\end{proposition}
\begin{proof}
The claims hold by \cite[Lemma~1.7, Lemma~2.9, Lemma~2.11, Proposition~2.12, Lemma~2.14, Proposition~2.15]{Picantin01}.
\end{proof}

\begin{remark}
The results from \cite{Picantin01} used in the proof of \autoref{P:QuasiCentre} do not depend on the notion of crossed products.
\end{remark}
\medskip

\noindent
We will only consider the prefix lattice, but the left-right symmetry
of our definitions means that analogous results hold for the suffix
ordering and the right normal form; cf.\ \autoref{transforming-expressions}.

\section{\hspace*{-1.3pt}Actions on the factors of \ZS{} products}\label{S:Actions}

In the situation of \autoref{D:ZappaSzepProduct}, the process of rewriting the $GH$-decomposition of an element into its $HG$-decomposition, or vice versa, defines a left-action and a right-action of~$H$ on~$G$, as well as a left-action and a right-action of $G$ on $H$; this section is devoted to analysing these actions.

Our treatment is analogous to~\cite{Brin05}, except that \autoref{D:ZappaSzepProduct} is symmetric under swapping the two factors of a \ZS{} product, whereas in~\cite{Brin05}, the factors play different roles.  The situation of \autoref{D:ZappaSzepProduct} is equivalent to requiring both $K=G\zs H$ and $K=H\zs G$ in the terminology of~\cite{Brin05}.

\begin{definition}\label{D:actions}
  Converting $HG$-decompositions into $GH$-decompositions, and vice
  versa, gives us the following maps:
  \begin{align*}
    H \times G & \longrightarrow G & 
    H \times G & \longrightarrow H \\
    (h, g)     & \longmapsto h \leftacts g &
    (h, g)     & \longmapsto h \rightacts g \\[1em] 
    G \times H & \longrightarrow H & 
    G \times H & \longrightarrow G \\
    (g, h)     & \longmapsto g \Leftacts h &
    (g, h)     & \longmapsto g \Rightacts h
  \end{align*}
  such that 
  $h g = (h \leftacts g)(h \rightacts g)$ and $g h = (g \Leftacts h)(g \Rightacts h)$.
  \medskip
  
  \noindent
  These definitions correspond to the following commutative diagrams:\medskip

  \hfill
  \begin{minipage}{3cm}
    \begin{tikzpicture}[scale=1.50]
      \draw [->] (0,0.95) -- (0,0.05) node [left] at (0,0.5) {$g$};
      \draw [->] (0.05,0) -- (0.95,0) node [below] at (0.5,0) {$h$};
      \draw [->] (0.05,1) -- (0.95,1) node [above] at (0.5,1) {$g\Leftacts h$};
      \draw [->] (1,0.95) -- (1,0.05) node [right] at (1,0.5) {$g\Rightacts h$};
    \end{tikzpicture}
  \end{minipage}
    \qquad\text{and}\qquad
  \begin{minipage}{3cm}
    \begin{tikzpicture}[scale=1.50]
      \draw [->] (0,0.95) -- (0,0.05) node [left] at (0,0.5) {$h\leftacts g$};
      \draw [->] (0.05,0) -- (0.95,0) node [below] at (0.5,0) {$h\rightacts g$};
      \draw [->] (0.05,1) -- (0.95,1) node [above] at (0.5,1) {$h$};
      \draw [->] (1,0.95) -- (1,0.05) node [right] at (1,0.5) {$g$};
    \end{tikzpicture}
  \end{minipage}
  \hfill\,
\end{definition}

\begin{remark}[Comparison with crossed products \cite{Picantin01}]\quad
\begin{enumerate}[itemindent=6mm,listparindent=\parindent,itemsep=1.5ex,parsep=0ex,leftmargin=0cm,topsep=1ex]
 \item
 Let~$G$ and~$H$ be cancellative, conical monoids with finitely many atoms, and let $\Theta_{1,2}: G\times H\to H$ and $\Theta_{2,1}: H\times G\to G$ satisfy the following conditions:
 \begin{enumerate}[leftmargin=18mm,label=(CP-\alph*)\hspace{2mm}]
  \item
    For any $g\in G$ and any $h\in H$, the maps $h'\mapsto \Theta_{1,2}(g,h')$ and $g' \mapsto \Theta_{2,1}(h,g')$ are bijections.
  \item
    For any $g,g_1,g_2\in G$ and $h,h_1,h_2\in H$, one has
    \begin{align*}
       & \Theta_{1,2}(g_1 g_2,h) = \Theta_{1,2}(g_2, \Theta_{1,2}(g_1,h))
       \\
       & \Theta_{2,1}(h_1 h_2,g) = \Theta_{2,1}(h_2, \Theta_{2,1}(h_1,g))
       \\
       & \Theta_{1,2}(g,h_1 h_2) = \Theta_{1,2}(g, h_1) \Theta_{1,2}(\Theta_{2,1}(h_1,g),h_2)
       \\
       & \Theta_{2,1}(h,g_1 g_2) = \Theta_{2,1}(h, g_1) \Theta_{2,1}(\Theta_{1,2}(g_1,h),g_2)
       \;.
    \end{align*}
 \end{enumerate}
 The crossed product $G\zs_\Theta H$ of~$G$ and~$H$ with respect to~$\Theta_{1,2}$ and $\Theta_{2,1}$ is defined~\cite{Picantin01} as the quotient of the free product of~$G$ and~$H$ by the relations corresponding to the following commutative diagrams:
 \begin{center}
   \begin{tikzpicture}[scale=1.50]
     \draw [->] (0,0.95) -- (0,0.05) node [left] at (0,0.5) {$g$};
     \draw [->] (0.05,0) -- (0.95,0) node [below] at (0.5,0) {$\Theta_{1,2}(g,h)$};
     \draw [->] (0.05,1) -- (0.95,1) node [above] at (0.5,1) {$h$};
     \draw [->] (1,0.95) -- (1,0.05) node [right] at (1,0.5) {$\Theta_{2,1}(h,g)$};
   \end{tikzpicture}
 \end{center}
 By~\cite[Proposition~3.6]{Picantin01}, every element of $G\zs_\Theta H$ admits a unique $GH$-de\-com\-pos\-ition and a unique $HG$-decomposition, that is, $G\zs_\Theta H$ is the internal \ZS{} product of its submonoids~$G$ and~$H$, and maps $\leftacts$, $\rightacts$, $\Leftacts$, $\Rightacts$ as in~\autoref{D:actions} exist.
 
 \item
 Conversely, we shall see that in an internal \ZS{} product $G\zs H$ where the factors~$G$ and~$H$ are cancellative, conical monoids with finitely many atoms, and if moreover common multiples exist in~$G$ and~$H$, the maps
   \[ g' \mapsto h \leftacts g'
      \qquad
      g' \mapsto g' \Rightacts h
      \qquad
      h' \mapsto h' \rightacts g
      \qquad
      h' \mapsto g \Leftacts h'
   \]
 are bijections for all $g\in G$ and $h\in H$ that satisfy properties analogous to~(CP-b) above
 (cf.\ \autoref{ZS-actions}, \autoref{action-injective}, \autoref{action-surjective}, \autoref{action-inverse-2}).
 In this case, it is possible to construct suitable maps $\Theta_{1,2}$ and $\Theta_{2,1}$ from the maps $\leftacts$, $\rightacts$, $\Leftacts$, $\Rightacts$, and one can identify $G\zs H$ with the crossed product $G\zs_\Theta H$; cf.\ \autoref{C:ZS-crossed}.
  
 However this need not be possible in general:
 \autoref{E:NotInjective} constructs a monoid that can be decomposed as the internal \ZS{} product of two cancellative, conical monoids with finitely many atoms that do not admit common multiples, but that cannot be written as a crossed product of these factors.
\end{enumerate}
\medskip

\noindent
While there are connections between the notions of crossed products and internal \ZS{} products, the two notions are not equivalent.
Moreover, there are conceptual differences: Internal \ZS{} products model the decomposition of a given monoid into components, whereas crossed products are a specific construction by which given monoids can be composed into a new monoid.
\end{remark}

\begin{example}\label{E:NotInjective}
%
%
Consider the monoid~$K$ defined by the monoid presentation
\[
  \Big\langle
    a_1,b_1,a_2,b_2
  \;\Big|\,
    a_1 a_2 = a_2 a_1,\, a_1 b_2 = a_2 b_1,\, b_1 a_2 = b_2 a_1,\, b_1 b_2 = b_2 b_1
  \Big\rangle^+
\]
and the submonoids~$G$ and~$H$ of~$K$ generated by $\{a_1,b_1\}$ respectively $\{a_2,b_2\}$.

The relations of~$K$ preserve the string over the alphabet $\{a,b\}$ obtained from a word in the atoms of~$K$ by ignoring generator subscripts, in particular the length of the word, as well as the number of atoms contained in~$G$ respectively~$H$.  Indeed, any two words in the atoms of~$K$ that yield the same string over the alphabet $\{a,b\}$ and involve the the same number of atoms contained in~$G$ (and thus the the same number of atoms contained in~$H$) are equivalent.
Thus, each element~$x$ of~$K$ can be identified with a pair consisting of a word~$w_x$ of length~$\ell_x$ over the alphabet $\{a,b\}$ and an integer $g_x\in\{0,\ldots,\ell_x\}$ giving the number of atoms contained in~$G$; one has $x\in G$ iff $g_x=\ell_x$ and $x\in H$ iff $g_x=0$.
Hence, the number of elements of word length~$\ell$ in~$G$, and also in~$H$, is $2^\ell$, and the number of elements of word length~$\ell$ in~$K$ is $(\ell+1)\cdot2^\ell$.
In particular, the submonoids~$G$ and~$H$ are free of rank~2, and each element of~$K$ admits a unique $GH$-decomposition and a unique $HG$-decomposition,
so one has $K=G\zs H$, where~$G$ and~$H$ are cancellative, conical monoids with finitely many atoms.
Note that neither in~$G$ nor in~$H$ do common multiples of elements exist.

From the commutative diagrams
\begin{equation}\label{eq:NotInjective}
  \raisebox{-6ex}{
  \begin{tikzpicture}[scale=0.8]
   \draw [->] (0,0.95) -- (0,0.05) node [left] at (0,0.5) {$a_1$};
   \draw [->] (0.05,0) -- (0.95,0) node [below] at (0.5,0) {$a_2$};
   \draw [->] (0.05,1) -- (0.95,1) node [above] at (0.5,1) {$a_2$};
   \draw [->] (1,0.95) -- (1,0.05) node [right] at (1,0.5) {$a_1$};
  \end{tikzpicture}
  \quad
  \begin{tikzpicture}[scale=0.8]
   \draw [->] (0,0.95) -- (0,0.05) node [left] at (0,0.5) {$a_1$};
   \draw [->] (0.05,0) -- (0.95,0) node [below] at (0.5,0) {$b_2$};
   \draw [->] (0.05,1) -- (0.95,1) node [above] at (0.5,1) {$a_2$};
   \draw [->] (1,0.95) -- (1,0.05) node [right] at (1,0.5) {$b_1$};
  \end{tikzpicture}
  \quad
  \begin{tikzpicture}[scale=0.8]
   \draw [->] (0,0.95) -- (0,0.05) node [left] at (0,0.5) {$b_1$};
   \draw [->] (0.05,0) -- (0.95,0) node [below] at (0.5,0) {$a_2$};
   \draw [->] (0.05,1) -- (0.95,1) node [above] at (0.5,1) {$b_2$};
   \draw [->] (1,0.95) -- (1,0.05) node [right] at (1,0.5) {$a_1$};
  \end{tikzpicture}
  \quad
  \begin{tikzpicture}[scale=0.8]
   \draw [->] (0,0.95) -- (0,0.05) node [left] at (0,0.5) {$b_1$};
   \draw [->] (0.05,0) -- (0.95,0) node [below] at (0.5,0) {$b_2$};
   \draw [->] (0.05,1) -- (0.95,1) node [above] at (0.5,1) {$b_2$};
   \draw [->] (1,0.95) -- (1,0.05) node [right] at (1,0.5) {$b_1$};
  \end{tikzpicture}
  }
\end{equation}
we see that one has for any $u\in\{a,b\}$, any $g\in\{a_1,b_1\}$, and any $h\in\{a_2,b_2\}$ the following identities:
\[
   u_2 \leftacts g = u_1
   \qquad
   h \rightacts u_1 = u_2
   \qquad
   u_1 \Leftacts h = u_2
   \qquad
   g \Rightacts u_2 = u_1
\]
That is, the maps $g \mapsto u_2 \leftacts g$, $h \mapsto h \rightacts u_1$, $h \mapsto u_1 \Leftacts h$,
and $g \mapsto g \Rightacts u_2$ are not injective in this case.
Moreover, as the commutative diagrams in \eqref{eq:NotInjective} are all the commutative diagrams involving two pairs of atoms, it is not possible to define maps $\Theta_{1,2}$ and $\Theta_{2,1}$ with the properties (CP-a) and (CP-b), so~$K$ cannot be identified with a crossed product of~$G$ and~$H$.
\end{example}

\subsection{Basic properties}\label{S::ActionProperties}

We start by noting some basic properties of the maps $\leftacts$, $\rightacts$, $\Leftacts$, $\Rightacts$.

\begin{lemma} \label{transforming-expressions}
  Consider the set of propositions built out of monoid operations, logical operations,
  quantifiers over $G$, $H$ and $K$, and the
  operations $\leftacts$, $\rightacts$, $\Leftacts$, $\Rightacts$.
  We can define two transformations of this set as follows.
  \begin{itemize}
  \item[$\sigma:$] Swap $G \longleftrightarrow H$, $\leftacts
    \longleftrightarrow \Leftacts$ and $\rightacts \longleftrightarrow
    \Rightacts$.
  \item[$\tau:$] Replace the monoids with their opposites,
    reversing all monoid expressions
    and all triangle operations:
    \begin{itemize}
      \item[$\circ$] $G \longrightarrow G^\op$,
        $H \longrightarrow H^\op$
        and $K\longrightarrow K^\op$.
      \item[$\circ$] $x \cdot y \longrightarrow y^\op \cdot x^\op$.
      \item[$\circ$] $h \leftacts g \longleftrightarrow g^\op \Rightacts h^\op$
        and $h \rightacts g \longleftrightarrow g^\op \Leftacts h^\op$.
    \end{itemize}
  \end{itemize}
  Then for any proposition $E$ we have $E \Longleftrightarrow
  \sigma(E) \Longleftrightarrow \tau(E)$.
\end{lemma}
%
%
%
\begin{proof}
  The equivalence of $E$ and $\sigma(E)$ is clear; if you swap the
  roles of $G$ and $H$ then you swap the definitions of the triangle
  operations.

  To see that $E$ is equivalent to $\tau(E)$, first observe that $K =
  G \zs H$ if and only if $K^\op = G^\op \zs H^\op$.  The
  anti-isomorphisms between $K$ and $K^\op$ transforms the equality $hg
  = (h \leftacts g)(h\rightacts g)$ to
  \[ g^\op h^\op = (h\rightacts g)^\op(h \leftacts g)^\op, \]
  but taking $H^\op G^\op$-decompositions gives us 
  \[ g^\op h^\op = (g^\op \Leftacts h^\op) (g^\op \Rightacts h^\op). \] 
  The uniqueness of $H^\op G^\op$-decompositions means we have the
  following equalities.
  \begin{align}
    (h\rightacts g)^\op &= (g^\op \Leftacts h^\op) &
    (h \leftacts g)^\op &= (g^\op \Rightacts h^\op) \label{eq:op}
  \end{align}

  If $E$ is simply an equality between two monoid-triangle
  expressions, i.e.\ $E$ is $x = y$, then by \eqref{eq:op} $\tau(x) =
  x^\op$ and $\tau(y) = y^\op$ so the equivalence of $E$ and
  $\tau(E)$ follows from the fact that $\op$ is a bijection.

  Logical conjunction and disjunction and the universal and
  existential quantifiers are unchanged by $\tau$, e.g.\ $\tau(A
  \wedge B) \equiv \tau(A) \wedge \tau(B)$. Therefore, the equivalence
  of $E$ and $\tau(E)$ follows by structural induction.
\end{proof}

\begin{lemma}[{\cite[Lemma 3.2]{Brin05}}] \label{ZS-actions}
  The maps $\leftacts$ and $\Leftacts$ define left actions, and
  $\rightacts$ and $\Rightacts$ define right actions:
  \begin{align*}
    (h_1 h_2) \leftacts g  &= h_1 \leftacts (h_2 \leftacts g) &
    (g_1 g_2) \Leftacts h  &= g_1 \Leftacts (g_2 \Leftacts h) \\
    h \rightacts (g_1 g_2) &= (h \rightacts g_1) \rightacts g_2 &
    g \Rightacts (h_1 h_2) &= (g \Rightacts h_1) \Rightacts h_2 
  \end{align*} 
  Moreover, the actions act on products as follows:
  \begin{align}\label{action-on-products}
      h \leftacts (g_1 g_2)
        &= (h \leftacts g_1) ((h \rightacts g_1) \leftacts g_2) \nonumber \\
      g \Leftacts (h_1 h_2) 
        &= (g \Leftacts h_1) ((g \Rightacts h_1) \Leftacts h_2) \\
      (h_1 h_2) \rightacts g 
        &= (h_1 \rightacts (h_2 \leftacts g)) (h_2 \rightacts g) \nonumber \\
      (g_1 g_2) \Rightacts h 
        &= (g_1 \Rightacts (g_2 \Leftacts h)) (g_2 \Rightacts h) \nonumber
  \end{align}
\end{lemma}
%
%

\begin{lemma}[{\cite[Corollary 3.3.1]{Brin05}}]
  The identity elements of the submonoids act trivially.
  \begin{align*}
    \id \leftacts g &= g   & h \rightacts \id &= h &
    \id \Leftacts h &= h   & g \Rightacts \id &= g 
  \end{align*}
\end{lemma}

\begin{lemma} \label{action-on-id}
  For all $g \in G$, $h \in H$ we have the following logical
  equivalences.
  \begin{align*}
    g = \id &\Longleftrightarrow h \leftacts  g = \id  &
    h = \id &\Longleftrightarrow h \rightacts g = \id  \\
    h = \id &\Longleftrightarrow g \Leftacts  h = \id  &
    g = \id &\Longleftrightarrow g \Rightacts h = \id 
  \end{align*}
\end{lemma}
\begin{proof}
  Consider the equation $hg = (h \leftacts g)(h \rightacts g)$.  If $g
  = \id$ then this is an element of $H$, so by the uniqueness of
  $GH$-decompositions $h \leftacts g = \id$.  Similarly, if $h
  \leftacts g = \id$ then this is also an element of $H$, so by the
  uniqueness of $HG$-decompositions $g = \id$.

  The remaining equivalences follow by \autoref{transforming-expressions}.
\end{proof}

\begin{lemma} \label{left-right-identity}
  For all $g \in G$ and $h \in H$:
  \begin{align*}
    (g \Leftacts h) \leftacts (g \Rightacts h)  &= g &
    (g \Leftacts h) \rightacts (g \Rightacts h) &= h \\ 
    (h \leftacts g) \Leftacts (h \rightacts g)  &= h &
    (h \leftacts g) \Rightacts (h \rightacts g) &= g  
  \end{align*}
\end{lemma}
\begin{proof}
  Rewriting a $GH$-decomposition as a $HG$- and then back as a
  $GH$-decom\-pos\-ition, we have
  \begin{align*}
    g h &= (g \Leftacts h) (g \Rightacts h) \\
        &= \left((g \Leftacts h) \leftacts (g \Rightacts h)\right)
           \left((g \Leftacts h) \rightacts (g \Rightacts h)\right)
    \;.
  \end{align*}
  
  \hfill
    \begin{tikzpicture}[scale=3.0]
      \draw [->] (0,0.95) -- (0,0.05) node [right] at (0,0.5) {$g$}
                                      node [left] at (0,0.5) {$(g\Leftacts h)\leftacts(g\Rightacts h)$};
      \draw [->] (0.05,0) -- (0.95,0) node [above] at (0.5,0) {$h$}
                                      node [below] at (0.5,0) {$(g\Leftacts h)\rightacts(g\Rightacts h)$};
      \draw [->] (0.05,1) -- (0.95,1) node [above] at (0.5,1) {$g\Leftacts h$};
      \draw [->] (1,0.95) -- (1,0.05) node [right] at (1,0.5) {$g\Rightacts h$};
    \end{tikzpicture}
  \hfill\,
  
  By the uniqueness of $GH$-decompositions we can the deduce the
  first two equations.  The second two can be shown in the same way.
\end{proof}

\subsection{Actions and the monoid structure}\label{S:ActionMonoid}

\begin{lemma} \label{ZS-factorial}
  Suppose that $K = G \zs H$ and that $H$ is conical.  Then, for all
  $x,y \in K$, $xy \in G$ implies that $x \in G$ and $y \in G$.
\end{lemma}
\begin{proof}
  Let $g = xy \in G$. Suppose that we have the following
  $GH$-decompositions of~$x$ and~$y$:
  \begin{align*}
    x &= g_x h_x &
    y &= g_y h_y
  \end{align*}
  Now
  \[ g = xy = g_x h_x g_y h_y 
       = g_x (h_x \leftacts g_y) (h_x \rightacts g_y) h_y \,, \]
  whence by the uniqueness of the $GH$-decomposition of $g$, we have the
  following:
  \begin{align*}
    g &= g_x (h_x \leftacts g_y) &
    \id &= (h_x \rightacts g_y) h_y
  \end{align*}
  So, as $H$ is conical, $h_x \rightacts g_y = \id = h_y$ and so, by
  \autoref{action-on-id}, $h_x = \id$.  Hence $x = g_x \in G$ and $y =
  g_y \in G$.
\end{proof}

\begin{lemma} \label{action-on-atoms}
  Suppose that $K = G \zs H$ and that $H$ is conical.  For all $h \in H$ we have that if $a \in G$
  is an atom then $h \leftacts a$ is an atom.
\end{lemma}
\begin{proof}
  Suppose that $h \leftacts a = x y$, that is, $ha = xy h'$
  where $h' = h \rightacts a$.  By \autoref{ZS-factorial}, $x, y \in
  G$, so we may apply \autoref{ZS-actions} to the action of $h'$ on $xy$, which gives us
  \[ a = (xy) \Rightacts h' 
       = (x \Rightacts (y \Leftacts h')) (y \Rightacts h') \; .  \]
  Now if $a$ is an atom, we have that either $(x \Rightacts (y
  \Leftacts h')) = \id$ or $(y \Rightacts h') = \id$. So,
  by \autoref{action-on-id}, either $x = \id$ or $y = \id$ holds.
\end{proof}

\begin{lemma}
  If $h \rightacts g = h$ then one has $h \leftacts (g^\ell) = (h \leftacts g)^\ell$ for all $\ell\in\NN$.
\end{lemma}
\begin{proof}
Using induction on $\ell$, we obtain
  \begin{align*}
    h \leftacts g^\ell &= h \leftacts (g g^{\ell-1})
     = (h \leftacts g)((h \rightacts g) \leftacts g^{\ell-1}) \\
    &= (h \leftacts g)(h \leftacts g^{\ell-1})
     = (h \leftacts g)^\ell  \; .
  \end{align*}
\end{proof}

\begin{lemma} \label{action-injective}
  Suppose that $K = G \zs H$, that $G$ is left-cancellative and that
  common multiples with respect to the prefix order exist in $G$ for
  every pair of elements.  Then $\leftacts$ acts by injections.
\end{lemma}
\begin{proof}
  Suppose that $h \leftacts g_1 = h \leftacts g_2$; we have to show
  that $g_1 = g_2$.  Let $g = h \leftacts g_1$.  There exist $h_1, h_2 \in H$ such
  that
  \begin{align}
    h g_1 &= g h_1 &
    h g_2 &= g h_2. \label{eq:injective}
  \end{align}
  Let $g_1 \bar g_1 = g_2 \bar g_2$ be a common multiple of $g_1$ and
  $g_2$ in $G$.  We have
  \begin{align*}
    h g_1 \bar g_1 = g h_1 \bar g_1 
      &= g(h_1 \leftacts \bar g_1) (h_1 \rightacts \bar g_1) \\
    h g_2 \bar g_2 = g h_2 \bar g_2 
      &= g(h_2 \leftacts \bar g_2) (h_2 \rightacts \bar g_2) \,,
  \end{align*}
  so uniqueness of $GH$-decompositions and the left-cancellativity of
  $G$ give us the following:
  \begin{align*}
      h_1 \leftacts \bar g_1 &= h_2 \leftacts \bar g_2 &
      h_1 \rightacts \bar g_1 &= h_2 \rightacts \bar g_2
  \end{align*}
  So $h_1 \bar g_1 = h_2 \bar g_2$, and uniqueness of
  $HG$-decompositions then yields $h_1 = h_2$.  Substituting this
  into~\eqref{eq:injective} gives $h g_1 = h g_2$, and using the uniqueness
  of $HG$-decompositions again, we obtain $g_1 = g_2$.
\end{proof}

\begin{remark}
 \autoref{E:NotInjective} shows that the condition that common multiples with respect to the prefix order exist in $G$ cannot be dropped.
\end{remark}

\begin{lemma} \label{action-surjective}
  Suppose that $K = G \zs H$, that $G$ is atomic and that $\leftacts$
  acts surjectively on the set of atoms.  Then $\leftacts$ acts
  surjectively on the whole of $G$.
\end{lemma}
\begin{proof}
  We need to show that given any $h \in H$ and $g \in G$ there exists
  $g' \in G$ such that $h \leftacts g' = g$.  As $G$ is atomic we may
  proceed by induction on the length of the longest decomposition of
  $g$ as a product of atoms.

  Suppose that $g = a g_1$ where $a$ is an atom of $G$.  As
  $\leftacts$ acts surjectively on the set of atoms, there exists an
  atom $b \in G$ such that $h \leftacts b = a$.  The longest atomic
  decomposition of $g_1$ must be at least one shorter than that for
  $g$, so by induction there exists $g_1' \in G$ such that
  \[ (h \rightacts b) \leftacts g_1' = g_1. \]
  Now
  \[
    h \leftacts (b g_1') = (h \leftacts b) ((h \rightacts b) \leftacts g_1')
                         = a g_1 = g.
  \]
  So $g' = b g_1'$ is the required element of $G$.
\end{proof}

\begin{lemma}[{cf.\ \cite[Lemma 3.12 (viii)]{Brin05}}] \label{injective-cancellative}
  Suppose that $K = G \zs H$, that $G$ and $H$ are left-cancellative
  and that $\leftacts$ acts by injections.  Then $K$ is
  left-cancellative.
\end{lemma}
\begin{proof}
  Suppose we have $x,y_1,y_2 \in K$ such that $x y_1 = x y_2$.  We
  have $GH$-decompositions $x = g_x h_x$, $y_1 = g_1 h_1$ and $y_2 =
  g_2 h_2$.  As
  \begin{align*}
    g_x \underbracket{h_x g_1} h_1 &= 
      g_x (h_x \leftacts g_1) (h_x \rightacts g_1) h_1 \\
    g_x \underbracket{h_x g_2} h_2 &= 
      g_x (h_x \leftacts g_2) (h_x \rightacts g_2) h_2 \,,
  \end{align*}
  uniqueness of $GH$-decompositions implies
  \begin{align*}
    g_x (h_x \leftacts g_1) &= g_x (h_x \leftacts g_2) \\
    (h_x \rightacts g_1) h_1 &= (h_x \rightacts g_2) h_2 \;.
  \end{align*}
  Left-cancellativity of $G$ implies that $h_x \leftacts g_1 =
  h_x \leftacts g_2$.  So, as $\leftacts$ acts injectively $g_1 =
  g_2$.  Left-cancellativity of $H$ then implies that $h_1 = h_2$,
  whence $y_1 = y_2$.
\end{proof}

\subsection{Submonoids acting by bijections}\label{S:ActionByBijections}

We will see in \autoref{S:ActionsForGarsideMonoids} that in the situations we are interested in the submonoids act on each other by bijections.
We analyse this special case in the remainder of this section.

\begin{definition}
  Suppose that $K = G \zs H$.  We say that the submonoids act on each other by bijections, if for all $h\in H$ the maps
  \[ g \mapsto h \leftacts g  \qquad\text{and}\qquad g \mapsto g \Rightacts h \]
  and for all $g\in G$ the maps
  \[ h \mapsto h \rightacts g  \qquad\text{and}\qquad h \mapsto g \Leftacts h \]
  are bijections.

  In this case, we denote the inverses of these maps as follows:
  \begin{align*}
    g \mapsto h^{-1} \leftacts g &:= (g \mapsto h \leftacts g)^{-1} &
    h \mapsto h \rightacts g^{-1} &:= (h \mapsto h \rightacts g)^{-1} \\
    h \mapsto g^{-1} \Leftacts h &:= (h \mapsto g \Leftacts h)^{-1} &
    g \mapsto g \Rightacts h^{-1} &:= (g \mapsto g \rightacts h)^{-1}
  \end{align*}
  Obviously, there are no \emph{elements} $g^{-1}$ and $h^{-1}$; this
  is just a notational convenience.
\end{definition}

\begin{lemma} \label{action-inverse}
  If $K = G \zs H$ and the submonoids act on each other by bijections
  then, for all $g \in G$ and $h \in H$, one has
  \begin{align*}
    h \rightacts (h^{-1} \leftacts g) &= g^{-1} \Leftacts h &
    (h \rightacts g^{-1}) \leftacts g &= g \Rightacts h^{-1} \\
    g \Rightacts (g^{-1} \Leftacts h) &= h^{-1} \leftacts g &
    (g \Rightacts h^{-1}) \Leftacts h &= h \rightacts g^{-1} \;.
  \end{align*}
\end{lemma}
\begin{proof}
  Suppose $g' = h^{-1} \leftacts g$.  So $h \leftacts g' = g$ and $h
  g' = g h'$ for some $h' \in H$.  From this we see that $h \rightacts
  g' = h'$ and $g \Leftacts h' = h$.  Substituting for $g'$ in the
  former and rearranging the latter we have $h \rightacts (h^{-1}
  \leftacts g) = h' = g^{-1} \Leftacts h$.  Hence the first equation
  holds.

  The remaining equations are shown in an analogous fashion.
\end{proof}

\begin{lemma} \label{action-inverse-2}
  If $K = G \zs H$ and the submonoids act on each other by bijections
  then, for all $g, g_1, g_2 \in G$ and $h, h_1, h_2 \in H$, the following identities hold:
  \begin{align*}
    (h_1 h_2)^{-1} \leftacts g &= h_2^{-1} \leftacts (h_1^{-1} \leftacts g) &
    (g_1 g_2)^{-1} \Leftacts h &= g_2^{-1} \Leftacts (g_1^{-1} \Leftacts h) \\
    h \rightacts (g_1 g_2)^{-1} &= (h \rightacts g_2^{-1}) \rightacts g_1^{-1} &
    g \Rightacts (h_1 h_2)^{-1} &= (g \Rightacts h_2^{-1}) \Rightacts h_1^{-1}
    \end{align*}
    \begin{align*}
    h^{-1} \leftacts (g_1 g_2)
    &=(h^{-1}\leftacts g_1)\left((g_1^{-1}\Leftacts h)^{-1}\leftacts g_2\right)\\
    g^{-1} \Leftacts (h_1 h_2)
    &=(g^{-1}\Leftacts h_1)\left((h_1^{-1}\leftacts g)^{-1}\Leftacts h_2\right)\\
    (h_1 h_2) \rightacts g^{-1}
    &=\left(h_1\rightacts(g\Rightacts h_2^{-1})^{-1}\right)(h_2\rightacts g^{-1})\\
    (g_1 g_2) \Rightacts h^{-1}
    &=\left(g_1\Rightacts(h\rightacts g_2^{-1})^{-1}\right)(g_2\Rightacts h^{-1})
  \end{align*}
\end{lemma}
\begin{proof}
  The first set of equations clearly hold as $\leftacts$,
  $\rightacts$, $\Leftacts$ and $\Rightacts$ define actions.

  Now consider the first of the second set of equations.  If we apply
  $h$ we have the following.
  \begin{align*}
    h \leftacts \Big(
      (h^{-1}\leftacts {}&g_1)\big((g_1^{-1}\Leftacts h)^{-1}\leftacts g_2\big)
    \Big) \\
    &= g_1
      \Big(
        \big(h \rightacts (h^{-1} \leftacts g_1)\big)
        \leftacts 
        \big((g_1^{-1}\Leftacts h)^{-1}\leftacts g_2\big)
      \Big) 
      &&\text{by \autoref{ZS-actions}} \\
    &= g_1
      \Big(
        (g_1^{-1} \Leftacts h)
        \leftacts 
        \big((g_1^{-1}\Leftacts h)^{-1}\leftacts g_2\big)
      \Big) 
      &&\text{by \autoref{action-inverse}} \\
    &= g_1 g_2
  \end{align*}
  Hence, as required, $h^{-1} \leftacts g_1 g_2 = (h^{-1} \leftacts
  g_1) \left((g_1^{-1} \Leftacts h)^{-1} \leftacts g_2\right)$.

  The other equations can be shown in the same way.
\end{proof}

\begin{proposition} \label{action-isomorphism}
  Suppose $K = G \zs H$, that $G$ and $H$ are cancellative, and that $G$ and $H$ act on each other by bijections.

  Then for all $g \in G$ and $h \in H$, the left actions are isomorphisms of the prefix order and the right actions are isomorphisms of the suffix order:
  \begin{align*}
    h \leftacts \cdot : (G, \prefix_G) &\xlongrightarrow{\sim} (G, \prefix_G) &
    \cdot \rightacts g : (H, \suffix_H) &\xlongrightarrow{\sim} (H, \suffix_H) \\
    g \Leftacts \cdot : (H, \prefix_H) &\xlongrightarrow{\sim} (H, \prefix_H) &
    \cdot \Rightacts h : (G, \suffix_G) &\xlongrightarrow{\sim} (G, \suffix_G)
  \end{align*}
\end{proposition}
\begin{proof}
  \autoref{ZS-actions} implies that these maps are poset morphisms:
  For example, if $g_1 \prefix_G g'$ then there exists $g_2$ such that
  $g' = g_1 g_2$, whence we obtain $h \leftacts g' = (h \leftacts g_1) \left((h
    \rightacts g_1) \leftacts g_2\right)$ and so $h \leftacts g_1
  \prefix_G h \leftacts g'$.

  Similarly, \autoref{action-inverse-2} implies that the inverses of
  these maps are poset morphisms: For example, if $g_1 \prefix_G g'$
  then there exists $g_2$ such that $g' = g_1 g_2$.
  So $h^{-1} \leftacts
  g' = (h^{-1} \leftacts g_1) \left((g_1^{-1} \Leftacts h)^{-1} \leftacts
    g_2\right)$ and thus $h^{-1} \leftacts g_1 \prefix_G h^{-1} \leftacts g'$.
\end{proof}

\begin{lemma}\label{action-under}
 Suppose $K = G \zs H$, that $G$ and $H$ are cancellative, and that $G$ and $H$ act on each other by bijections.

 Then for all $g_1,g_2 \in G$ such that $g_1\join g_2$ exists in $G$ and all $h \in H$ one has
 the following:
 \begin{align*}
   h \leftacts (g_1\under g_2)
        &= \big(g_1\Rightacts h^{-1}\big)
           \,\under\,
           \Big(\big(h\rightacts g_1^{-1}\big)\leftacts g_2 \Big) \\[1ex]
   h^{-1} \leftacts (g_1\under g_2)
        &= \big(g_1\Rightacts h\big)
           \,\under\,
           \Big((g_1 \Leftacts h)^{-1}\leftacts g_2 \Big)  
 \end{align*}
\end{lemma}
\begin{proof}
 For any $h'\in H$, one has $h'\leftacts (g_1\join g_2) = (h'\leftacts g_1)\join (h'\leftacts g_2)$ by \autoref{action-isomorphism}.
 On the other hand, using \autoref{ZS-actions}, one has
 $h'\leftacts (g_1\join g_2)
   = h'\leftacts \big(g_1(g_1\under g_2)\big)
   = (h'\leftacts g_1)\big((h'\rightacts g_1)\leftacts (g_1\under g_2)\big)$.
 As $G$ is cancellative, these imply
 $(h'\leftacts g_1)\under(h'\leftacts g_2) = (h'\rightacts g_1)\leftacts (g_1\under g_2)$.
 The first claim follows setting $h'=h\rightacts g_1^{-1}$ and simplifying
 $(h\rightacts g_1^{-1})\leftacts g_1 = g_1\Rightacts h^{-1}$ using \autoref{action-inverse}.
 
 The second claim is shown in the same way, using \autoref{action-inverse-2} instead of \autoref{ZS-actions}
 and \autoref{left-right-identity} instead of \autoref{action-inverse}.
\end{proof}

\begin{lemma}\label{PosetsWellDefined}
 Suppose $K = G \zs H$, that $G$ and $H$ are cancellative and conical, and that $G$ and $H$ act on each other by injections.
 
 Then $(K,\prefix_K)$ and $(K,\suffix_K)$ are posets, and the restrictions of $\prefix_K$ and $\suffix_K$ to $G\times G$ and $H\times H$ coincide with $\prefix_G$, $\suffix_G$, $\prefix_H$ and $\suffix_H$ respectively:
 \[
   \begin{array}{r@{\;\;=\;\;}l@{\qquad\qquad}r@{\;\;=\;\;}l}
     \prefix_K\!{}_{|_{G\times G}} & \prefix_G
      &
     \suffix_K\!{}_{|_{G\times G}} & \suffix_G
      \\[2ex]
     \prefix_K\!{}_{|_{H\times H}} & \prefix_H
      &
     \suffix_K\!{}_{|_{H\times H}} & \suffix_H 
   \end{array}
 \]
\end{lemma}
\begin{proof}
  By \autoref{injective-cancellative} and \autoref{transforming-expressions}, the monoid $K$ is cancellative, hence~$\prefix$ and~$\suffix$ define partial orders.
  
  \autoref{ZS-factorial} implies that $G$ and $H$ are closed under $\factor$, which implies the rest of the claim.
  For instance, if $g_1 \prefix_K g_2$ holds for $g_1,g_2\in G$, then there exists $x\in K$ such that $g_1 x = g_2$.  As $G$ is closed under $\factor$, one has $x\in G$ and thus $g_1 \prefix_G g_2$.
  Conversely, $g_1 \prefix_G g_2$ trivially implies $g_1 \prefix_K g_2$.
\end{proof}

\medskip

\noindent
In the situation of \autoref{PosetsWellDefined}, we will in the following just write $\prefix$ and $\suffix$ instead of $\prefix_K$, $\suffix_K$, $\prefix_G$, $\suffix_G$, $\prefix_H$ and $\suffix_H$.
\medskip

\begin{proposition} \label{ZS-lcm}
  Suppose $K = G \zs H$, that $G$ and $H$ are cancellative and conical, and that $G$ and $H$ act on each other by bijections.

  Then for all $g \in G$ and $h\in H$, we have
  \[ g \join h = g h' = h g' =g'\rightjoin h' \]
  where $h' = g^{-1} \Leftacts h$ and $g' = h^{-1} \leftacts g$.
  Moreover, if $g_1,g_2\in G$ and $h_1,h_2\in H$ satisfy $g_1\join h_1 =g_2\join h_2$ or $g_1\rightjoin h_1 =g_2\rightjoin h_2$, then $g_1=g_2$ and $h_1=h_2$.
\end{proposition}
\begin{proof}
  We are in the situation of \autoref{PosetsWellDefined}.
  
  First we will show that $gh' = hg'$.  As $h' = g^{-1} \Leftacts h$, we have
  $g \Leftacts h' = h$ and, using \autoref{action-inverse},
  $
    gh' = h (g \Rightacts h')
        = h (g \Rightacts (g^{-1} \Leftacts h))
        = h (h^{-1} \leftacts g)
        = h g'
  $.
  Therefore, $gh' = hg'$ is a common upper bound of $g$ and $h$ with respect to $\prefix$.

  Now assume we have $g_1,g_2\in G$ and $h_1,h_2\in H$ such that
  $gh_1g_1=hg_2h_2$ is a common upper bound of $g$ and $h$ with
  respect to $\prefix$.  As we have
  \begin{align*}
    g h_1 g_1 &= g (h_1 \leftacts g_1) (h_1 \rightacts g_1) \quad\text{and} \\
    h g_2 h_2 &= (h \leftacts g_2) (h \rightacts g_2) h_2 \;,
  \end{align*}
  uniqueness of $GH$-decompositions implies that $g (h_1\leftacts
  g_1) = (h \leftacts g_2)$.  Acting by $h^{-1}$ on both sides of this
  equality and applying \autoref{action-inverse-2}, we obtain $g' =
  h^{-1} \leftacts g \prefix g_2$, and thus $h g' \prefix h g_2
  \prefix h g_2 h_2$.

  Finally, let $g_1,g_2\in G$ and $h_1,h_2\in H$ satisfy $g_1\join h_1 =g_2\join h_2 = x\in K$.
  By the first part of the proposition, we have $x=g_1 (g_1^{-1} \Leftacts h_1) = g_2 (g_2^{-1} \Leftacts h_2)$.
  Uniqueness of $GH$-decompositions implies $g_1=g_2$ and $g_1^{-1} \Leftacts h_1 = g_2^{-1} \Leftacts h_2$, and this in turn implies $h_1=h_2$, as the action of~$G$ on~$H$ is by bijections.

  The claims for $\rightjoin$ are analogous.
\end{proof}

\begin{remark}
A result for crossed products that is analogous to the following \autoref{ZS-lattice-isomorphism} is~\cite[Proposition~3.12]{Picantin01};
in the light of~\autoref{C:ZS-crossed}, indeed~\autoref{ZS-lattice-isomorphism} follows from~\cite[Proposition~3.12]{Picantin01} if~$K$ is a Garside monoid.
\end{remark}

\begin{theorem}\label{ZS-lattice-isomorphism}
  Suppose $K = G \zs H$, that $G$ and $H$ are cancellative and conical, and that $G$ and $H$ act on each other by bijections.

  The map $G \times H \to K$ given by $(g,h) \mapsto g \join h$ is a poset isomorphism
  $(G, \prefix_G)\times(H, \prefix_H) \to (K, \prefix_K)$.
  
  Similarly, the map $G \times H \to K$ given by $(g,h) \mapsto g \rightjoin h$ is a poset isomorphism
  $(G, \suffix_G)\times(H, \suffix_H) \to (K, \suffix_K)$.
\end{theorem}
\begin{proof}
  We are in the situation of \autoref{PosetsWellDefined}, so we will drop the subscripts of the partial orders.
  
  By \autoref{ZS-lcm}, we can write any $x \in K$ in a unique way as $x = g_1 \join h_2$,
  where $g_1 h_1$ and $h_2 g_2$ are the $GH$-, respectively, $HG$-decompositions of~$x$.
  Hence the map $(g,h) \mapsto g \join h$ is a bijection.
  \medskip

  \begin{claim*}
    If $g_1, g_2 \in G$ and $h_1, h_2 \in H$ are such that $g_1 \prefix
    g_2$ and $h_1 \prefix h_2$ then $g_1 \join h_1 \prefix g_2 \join
    h_2$.
  \end{claim*}

  If $g_1 \prefix g_2$ and $h_1 \prefix h_2$, then there
  are $g_3 \in G$ and $h_3 \in H$ such that $g_1 g_3 = g_2$ and
  $h_1 h_3 = h_2$.  Now consider the following, where $g_1 \join h_1 =
  g_1 h_1' = h_1 g_1'$:
  \begin{align*}
    (g_1 \join h_1) &
    \left((h_1'^{-1} \leftacts g_3) \join (g_1'^{-1} \Leftacts h_3)\right) \\
    &= g_1 h_1' (h_1'^{-1} \leftacts g_3) \left(
      (h_1'^{-1} \leftacts g_3)^{-1} \Leftacts (g_1'^{-1} \Leftacts h_3) 
    \right) \\
    &= \underbrace{g_1 g_3}_{\in G}
    \underbrace{\left( h_1' \rightacts (h_1'^{-1} \leftacts g_3) \right) \left(
        (h_1'^{-1} \leftacts g_3)^{-1} \Leftacts (g_1'^{-1} \Leftacts h_3) 
      \right)}_{\in H} \\[1ex]
%
%
    (g_1 \join h_1) &
    \left((h_1'^{-1} \leftacts g_3) \join (g_1'^{-1} \Leftacts h_3)\right) \\
    &= h_1 g_1' (g_1'^{-1} \Leftacts h_3) \left(
      (g_1'^{-1} \Leftacts h_3)^{-1} \leftacts (h_1'^{-1} \leftacts g_3) 
    \right) \\
    &= \underbrace{h_1 h_3}_{\in H}
    \underbrace{\left( g_1' \Rightacts (g_1'^{-1} \Leftacts h_3) \right) \left(
        (g_1'^{-1} \Leftacts h_3)^{-1} \leftacts (h_1'^{-1} \leftacts g_3) 
      \right)}_{\in G}
  \end{align*}
  By the uniqueness of $GH$- and $HG$-decompositions and \autoref{ZS-lcm}, we thus have $(g_1 \join h_1) \left((h_1'^{-1} \leftacts g_3) \join (g_1'^{-1}
    \Leftacts h_3)\right) = g_1 g_3 \join h_1 h_3 = g_2 \join h_2$.
  Hence \mbox{$g_1 \join h_1 \prefix g_2 \join h_2$} and so the claim holds.
  \medskip
  
  \begin{claim*}
    If $g_1, g_2 \in G$ and $h_1, h_2 \in H$ are such that $g_1 \join h_1
    \prefix g_2 \join h_2$ then $g_1 \prefix g_2$ and $h_1 \prefix
    h_2$.
  \end{claim*}

  If $g_1 \join h_1 \prefix g_2 \join h_2$, there are
  $g_3 \in G$ and $h_3 \in H$ such that $(g_1 \join h_1)(g_3 \join
  h_3) = g_2 \join h_2$.

  Now consider the following.
  \begin{align*}
    (g_1 \join h_1)(g_3 \join h_3)
    &= g_1 (g_1^{-1} \Leftacts h_1) g_3 (g_3^{-1} \Leftacts h_3) \\
    &= \underbrace{
      g_1 \left((g_1^{-1} \Leftacts h_1) \leftacts g_3\right)
    }_{\in G}
    \underbrace{
      \left((g_1^{-1} \Leftacts h_1) \rightacts g_3\right)
      (g_3^{-1} \Leftacts h_3)
    }_{\in H} \\[1ex]
%
%
    (g_1 \join h_1)(g_3 \join h_3)
    &= h_1 (h_1^{-1} \leftacts g_1) h_3 (h_3^{-1} \leftacts g_3) \\
    &= \underbrace{
      h_1 \left((h_1^{-1} \leftacts g_1) \Leftacts h_3\right)
    }_{\in H}
    \underbrace{
      \left((h_1^{-1} \leftacts g_1) \Rightacts h_3\right)
      (h_3^{-1} \leftacts g_3)
    }_{\in G}
  \end{align*}
  Therefore $(g_1 \join h_1)(g_3 \join h_3) = g_1 \left((g_1^{-1}
    \Leftacts h_1) \leftacts g_3\right) \join h_1 \left((h_1^{-1}
    \leftacts g_1) \Leftacts h_3\right)$.  So $g_2 = g_1
  \left((g_1^{-1} \Leftacts h_1) \leftacts g_3\right)$ and $h_2 = h_1
  \left((h_1^{-1} \leftacts g_1) \Leftacts h_3\right)$.  Hence $g_1
  \prefix g_2$ and $h_1 \prefix h_2$ and so the claim holds.
  \medskip 

  \noindent
  We have shown that the map $(g,h)\mapsto g\join h$ is invertible and
  that both this map and its inverse preserve the ordering.  Therefore
  it is an isomorphism between the respective posets.
  
  The claim for the map $(g,h)\mapsto g\rightjoin h$ is shown analogously.
\end{proof}

\begin{lemma}\label{L:ZS-under}
  Suppose $K = G \zs H$, that $G$ and $H$ are cancellative and conical, and that $G$ and $H$ act on each other by bijections.
  
  Then, for all $g_1,g_2 \in G$ such that $g_1\join g_2\in G$ exists and for all $h_1,h_2 \in H$ such that $h_1\join h_2\in H$ exists, the elements $g_1 \join h_1$ and $g_2\join h_2$ of $K$ admit a $\prefix$-least common upper bound in $K$, and one has
  \[ (g_1 \join h_1)\under (g_2\join h_2) 
     = \bigg(\big( g_1^{-1}\Leftacts h_1 \big)^{-1} \leftacts (g_1\under g_2)\bigg)\,\join\,
       \bigg(\big( h_1^{-1}\leftacts g_1 \big)^{-1} \Leftacts (h_1\under h_2)\bigg) . \]
\end{lemma}
\begin{proof}
Let $g'=g_1\under g_2$ and $h'=h_1\under h_2$, so $g_1\join g_2 = g_1 g'$ and $h_1\join h_2 = h_1 h'$.
Using \autoref{ZS-lattice-isomorphism}, \autoref{ZS-lcm}, \autoref{action-inverse} and \autoref{action-inverse-2}, we obtain
\begin{align*}
 (g_1&\join h_1) \join (g_2\join h_2)
    = (g_1\join g_2) \join (h_1\join h_2)
    = g_1 g' \Big( (g_1 g')^{-1} \Leftacts (h_1h') \Big) \\
   &= g_1 g' \Big( g'^{-1} \Leftacts \big( g_1^{-1} \Leftacts (h_1h') \big) \Big) \\
   &= g_1 \big(g_1^{-1} \Leftacts (h_1 h')\big)
          \Big( g' \Rightacts \Big( g'^{-1} \Leftacts \big(g_1^{-1} \Leftacts (h_1 h') \big) \Big) \Big ) \\
   &= \underbrace{g_1 \big(g_1^{-1} \Leftacts h_1\big)}_{= g_1\join h_1}
          \underbrace{\Big( \big(h_1^{-1} \leftacts g_1\big)^{-1} \Leftacts h' \Big)}_{\in H}
          \underbrace{\Big( \big( g_1^{-1} \Leftacts (h_1 h') \big)^{-1} \leftacts g' \Big)}_{\in G}
\end{align*}
and likewise
\begin{align*}
 (g_1&\join h_1) \join (g_2\join h_2)
    = (g_1\join g_2) \join (h_1\join h_2)
    = h_1 h' \Big( (h_1 h')^{-1} \leftacts (g_1g') \Big) \\
   &= \ldots = \underbrace{h_1 \big(h_1^{-1} \leftacts g_1\big)}_{= g_1\join h_1}
          \underbrace{\Big( \big(g_1^{-1} \Leftacts h_1\big)^{-1} \leftacts g' \Big)}_{\in G}
          \underbrace{\Big( \big( h_1^{-1} \leftacts (g_1 g') \big)^{-1} \Leftacts h' \Big)}_{\in H}
  \;.
\end{align*}
Thus, as $K$ is cancellative, applying \autoref{ZS-lcm} yields
\begin{multline*}
(g_1\join h_1) \join (g_2\join h_2) = \\ 
(g_1\join h_1) \bigg(
\Big(\big( g_1^{-1}\Leftacts h_1 \big)^{-1} \leftacts (g_1\under g_2)\Big)\,\join\,
\Big(\big( h_1^{-1}\leftacts g_1 \big)^{-1} \Leftacts (h_1\under h_2)\Big) \bigg).
\end{multline*}
\end{proof}

\section{Actions in the case of Garside monoids}
\label{S:ActionsForGarsideMonoids}

In this section, we analyse the actions of the factors of a \ZS{} product on one another in the case that the product is a Garside monoid, or that both of the factors are Garside monoids.
Using these results, we prove that a \ZS{} product $K=G\zs H$ of monoids is a Garside monoid if and only if both~$G$ and~$H$ are Garside monoids.

\begin{lemma} \label{action-bijections}
  If $K = G \zs H$ is a Garside monoid then the submonoids act on each other by
  bijections.
\end{lemma}
\begin{proof}
  We will first show that the maps are injective.  So suppose that $h
  \leftacts g_1 = h \leftacts g_2 = g$; we need to show that $g_1 = g_2$.

  Let $h_1 = h \rightacts g_1$ and $h_2 = h \rightacts g_2$.  So we have
  \begin{align} \label{bijective1}
    h g_1 &= g h_1 \qquad\text{and}\qquad
    h g_2 = g h_2 \;.
  \end{align}

  First consider the case when $g_1 \meet g_2 = \id$.  Taking the GCD
  of the two elements in \eqref{bijective1} gives
  \begin{align*}
    h g_1 \meet h g_2 = h (g_1 \meet g_2) = h \qquad\text{and}\qquad
    g h_1 \meet g h_2 = g (h_1 \meet h_2) \;.
  \end{align*}
  Uniqueness of $GH$-decompositions implies $g = \id$ (and
  $h = h_1 \meet h_2$), and uniqueness of the $HG$-decompositions
  in \eqref{bijective1} then implies $g_1 = \id = g_2$.

  Now suppose that $g_1 \meet g_2 \ne \id$, so
  \begin{align*}
    g_1 = (g_1 \meet g_2) \bar g_1 \qquad\text{and}\qquad
    g_2 = (g_1 \meet g_2) \bar g_2
  \end{align*}
  for some $\bar g_1, \bar g_2 \in G$ with $\bar g_1 \meet \bar g_2 = \id$.

  We can now apply the formula for the action on a product from
  \autoref{ZS-actions}:
  \begin{align*}
    h \leftacts g_1 &= h \leftacts (g_1 \meet g_2) \bar g_1 \\
                    &= \left(h \leftacts (g_1\meet g_2)\right)
                       \left(\left(h \rightacts (g_1 \meet g_2)\right)
                         \leftacts \bar g_1\right) \\[0.5em]
    h \leftacts g_2 &= h \leftacts (g_1 \meet g_2) \bar g_2 \\
                    &= \left(h \leftacts (g_1\meet g_2)\right)
                       \left(\left(h \rightacts (g_1 \meet g_2)\right)
                         \leftacts \bar g_2\right)
  \end{align*}
  Cancellativity of $K$ means that
  $ h' \leftacts \bar g_1 = h' \leftacts \bar g_2 $, 
  where $h' = h \rightacts (g_1 \meet g_2)$.  As $\bar g_1 \meet \bar
  g_2 = \id$ we can apply the first case to deduce that $\bar g_1 =
  \bar g_2$ and so $g_1 = g_2$.

  Similar arguments show that the other maps are injective, so it
  remains to show that the maps are surjective.

  \commS{I think this is the only place where we use the finiteness
    condition.  If we can avoid using it then I think everything
    generalizes to quasi-Garside monoids.}
  \commV{This was commented on when I gave the talk in the Garside theory working group at Paris.
    The general feeling there seemed to be that we actually do use this hypothesis here.}
  First note that, by \autoref{action-on-atoms}, the maps take atoms
  to atoms.  So, as the sets of atoms are finite, the maps are
  bijections on these sets.  The surjectivity of the maps then follows from \autoref{action-surjective} and \autoref{transforming-expressions}.
\end{proof}

\begin{remark}
An analogous result for crossed products is \cite[Lemma 3.2]{Picantin01}.
\end{remark}

\begin{corollary}\label{C:ZS-crossed}
 If $K$ is a Garside monoid with submonoids~$G$ and~$H$, one has $K = G\zs H$ if and only if~$K$ can be written as a crossed product of the monoids~$G$ and~$H$.
\end{corollary}
\begin{proof}
If~$K$ is a crossed product of~$G$ and~$H$, then every element of~$K$ has a unique $GH$-decomposition and a unique $HG$-decomposition by~\cite[Proposition~3.6]{Picantin01}, so one has $K=G\zs H$.

If $K = G\zs H$, then~\autoref{action-bijections} implies that~$G$ and~$H$ act on each other by bijections and we can define
\[
\Theta_{1,2}(g,h) = g^{-1} \Leftacts h
\qquad \mbox{and}\qquad
\Theta_{2,1}(h,g) = h^{-1} \leftacts g
\]
for $g\in G$ and $h\in H$. 
The maps $\Theta_{1,2}$ and $\Theta_{2,1}$ satisfy (CP-a) and (CP-b) by~\autoref{action-inverse-2}, so~$K$ is a crossed product of~$G$ and~$H$.
\end{proof}

\begin{remark}
\autoref{C:ZS-crossed} says that, in the context of Garside monoids, the notion of crossed products in the sense of~\cite{Picantin01} and the notion of (external or internal) \ZS{}-products in the sense of~\cite{Brin05} are equivalent.
\autoref{E:NotInjective} shows that this is not true in general.
\end{remark}

\begin{theorem} \label{ZS-parabolic} 
  If $K=G\zs H$ is a Garside monoid then $G$ and $H$ are parabolic submonoids of $K$.
  In particular, $G$ and $H$ are Garside monoids.
\end{theorem}
\begin{proof}
  Let $d_G d_H$ and $e_H e_G$ be the $GH$-, respectively,
  $HG$-decompositions of~$\Delta$.

  Suppose that $x \in \Simples \intersect G$ is a simple element which
  lies in $G$.  Now, as $x$ is a simple element, there is
  $\complement x \in K$ such that $x \complement x = \Delta$.  Let $gh$ be
  the $GH$-decomposition of $\complement x$.  So we have
  \[ x g h = \Delta \;. \]
  As $x \in G$, the uniqueness of $GH$-decompositions means that
  \[ x g = d_G \qquad \text{and} \qquad h = d_H \;. \]
  Hence $x$ is a prefix of $d_G$.  Since $d_G$ is a simple element and
  a member of the submonoid $G$, we have that $d_G$ is the
  $\prefix$-LCM of the intersection of $\Simples$ and $G$.  A similar
  argument shows that $e_G$ is the $\suffix$-LCM of the same set.
  \begin{align}
    d_G &= \leftJoin (\Simples \intersect G) &
    e_G &= \rightJoin (\Simples \intersect G) \label{ZS-LCM}
  \end{align}
  Now observe that $e_G \in \Simples \intersect G$ and hence $e_G
  \prefix d_G$.  We also have that $d_G \in \Simples \intersect G$ and
  so $e_G \suffix d_G$.  Together these imply that $d_G = e_G$.

  If $x$ is a prefix of $d_G$ then it is a simple element and, by
  \autoref{ZS-factorial}, an element of $G$.  Therefore $x$ is an
  element of $\Simples \intersect G$.  So, by \eqref{ZS-LCM}, $x$ is a
  suffix of $e_G = d_G$.  Similarly, every suffix of $d_G$ is also a
  prefix of $d_G$.  Therefore $d_G$ is a balanced element with 
  $\Div(d_G) = \Simples \intersect G$.
  
  Every element of $K$ can be written as a product of atoms, so, by
  \autoref{ZS-factorial}, $G$ is generated by the atoms of $K$ which
  lie in $G$.  Every atom in $G$ is clearly in $\Simples \intersect
  G$, hence $G$ is generated by the divisors of $d_G$.  Therefore, $G$
  is a parabolic submonoid and $d_G$ is a Garside element.

  The same argument with the roles of $G$ and $H$ reversed shows that
  $H$ is also a parabolic submonoid and that $d_H = e_H$ is a
  Garside element.
\end{proof}

\begin{remark}
The proof of \autoref{ZS-parabolic} shows that decomposing a Garside element of a Garside monoid $K=G\zs H$
gives Garside elements for $G$ and $H$.  However, not every pair of Garside elements for the submonoids $G$ and $H$ can be produced this way, as \autoref{E:SemidirectProduct} shows.
\end{remark}

\begin{proposition}\label{LocalDeltas}
Let $K=G\zs H$ be a Garside monoid and let $g\in G$.
Then $\Delta_g^{K} = \Join \{ x\under g : x \in K \} \in G$.
\end{proposition}
\begin{proof}
For $x=g_1\join h_1$ with $g_1\in G$ and $h_1\in H$, write $g_1\join g = g_1g_2$. Then
\begin{align*}
      g\join x
       = &\ (g_1\join g)\join h_1 
       = (g_1 g_2)\join h_1 \\
     & = h_1 \big(h_1^{-1}\leftacts (g_1g_2)\big)
       = h_1(h_1^{-1}\leftacts g_1) g_3
       = x g_3
\end{align*}
with $g_3 = (g_1^{-1}\Leftacts h_1)^{-1} \leftacts g_2$, that is, $x\under g = g_3\in G$.
As $x$ was arbitrary and $G$ is a parabolic submonoid by \autoref{ZS-parabolic}, we have $\Delta_g^{K}\in G$.
\end{proof}

\begin{example}\label{E:SemidirectProduct}
Consider the monoid $K = \langle a,b,c \,|\, ab=ba, ac=cb, bc=ca\rangle^+$ with the submonoids $G_1 = \langle a \rangle^+$, $G_2 = \langle b \rangle^+$, $G=\langle a,b \rangle^+$ and $H = \langle c \rangle^+$.
(That is, $G \cong \NN_0^2$ and $K \cong \NN_0^2 \rtimes \langle c \rangle^+$ where the action of $c$ on~$\NN_0^2$ is given by swapping the coordinates.)  Clearly, the monoid $K$ is a \ZS{} product of the submonoids $G$ and $H$.  Moreover, $K$, $G$ and $H$ are Garside monoids whose minimal Garside elements are $\Delta_K=abc$, $\Delta_G=ab$, respectively $\Delta_H=c$.
\begin{enumerate}[itemindent=6mm,listparindent=\parindent,itemsep=\parsep,leftmargin=0cm]
\item
We see that not every pair of Garside elements for~$G$ and~$H$ can be obtained by decomposing a Garside element of~$K$:
The element $\Delta'_G=a^2b$ is also a Garside element for the monoid $G$, yet $\Delta'_G \Delta_H=a^2bc=cab^2$ is not balanced (and not equal to $\Delta_H\Delta'_G$) and so cannot be a Garside element for the monoid $K$.

\item
Moreover, although for $g\in G$ one has $\Delta_g^{K} = \Join \{ x\under g : x \in K \} \in G$ by \autoref{LocalDeltas}, in general $\Delta_g^{K} \neq \Delta_g^{G} = \Join \{ x\under g : x \in G \}$:
For $x\in G$ one has $x\under a=\id $ if $a\prefix x$ and $x\under a=a$ otherwise.  Thus, $\Delta_a^{G}=a$.  However, $c\under a=b$, so $\Delta_a^{K}\neq\Delta_a^{G}$, although both are elements of $G$.  (In fact, $\Delta_a^{K}=ab=\Delta_G$.)

\item
The example also shows that forming \ZS{} products is not associative:
We have $K = (G_1 \zs G_2) \zs H$,
but any parabolic submonoid containing both $b$ and $c$ also must contain $a\prefix cb$, so it is not true that $K = G_1 \zs (G_2 \zs H)$.
\end{enumerate}
\end{example}

\begin{theorem} \label{ZS-of-Garside-is-Garside}
  Suppose that $K = G \zs H$ and that $G$ and $H$ are Garside monoids.
  Then $K$ is a Garside monoid.
\end{theorem}
\begin{proof}
  We write $\Delta_g$ to mean $\Delta_g^G$ for $g\in G$ and $\Delta_h$ to mean $\Delta_h^H$ for $h\in H$.

  By \autoref{action-injective} and \autoref{transforming-expressions}, the monoids act on each other by injections.
  Let $\Atoms_G$ and $\Atoms_H$ denote the sets of atoms of $G$ respectively $H$.
  By \autoref{action-on-atoms}, the actions act on the sets $\Atoms_G$ respectively $\Atoms_H$, so as these sets are finite, the actions are surjective on the sets of atoms, and thus the actions act surjectively on the whole of the submonoids by \autoref{action-surjective}.
  
  We are in the situation of \autoref{PosetsWellDefined}.  
  For $g \in G$ and $h \in H$, \autoref{ZS-lcm} yields that $gh' = hg'$ is the $\prefix$-LCM of $g$ and $h$ in $K$, where $h' = g^{-1} \Leftacts h$ and $g' = h^{-1} \leftacts g$.
  Moreover, by \autoref{ZS-lattice-isomorphism}, the map $(g,h) \mapsto g \join h$ is a poset isomorphism, hence $(K,\prefix)$ is a lattice.
  Likewise, using the map $(g,h) \mapsto g \rightjoin h$, one has that $(K,\suffix)$ is a lattice.
  
  As $G$ and $H$ are closed under $\factor$ by \autoref{ZS-factorial}, the set of atoms of $K$ is $\Atoms = \Atoms_G\cup\Atoms_H$.
  As every element of $K$ has a $GH$-decomposition and both $G$ and $H$ are atomic, $K$ is generated by $\Atoms_G\cup\Atoms_H=\Atoms$.
  Suppose $k=gh$ with $g\in G$ and $h\in H$.
  By \autoref{action-on-atoms}, we can rewrite each expression for $k$ as a product of atoms of $K$ as a $GH$-decomposition without changing its length.  Hence $||k||_{\Atoms} = ||g||_{\Atoms_G} + ||h||_{\Atoms_H} < \infty$, so $K$ is atomic.
  

  Define $D_G := \Join_{a\in\Atoms_G} \Delta_a = \Join \big\{ g\under a : g\in G,\, a\in\Atoms_G \big\}$.
  By \autoref{P:QuasiCentre}, we have that $D_G = \Delta_{\Join\Atoms_G}$ is balanced.
  Moreover, for any $h\in H$ we have by \autoref{action-isomorphism}, \autoref{action-under} and \autoref{action-on-atoms}
  \begin{align*}
    h \leftacts D_G 
     &= \Join \Big\{ h \leftacts (g\under a) : g\in G,\, a\in\Atoms_G \Big\} \\
     &= \Join \Big\{ (\underbrace{g\Rightacts h^{-1}}_{\in G})
                     \under \big( \underbrace{(h\rightacts g^{-1})\leftacts a}_{\in\Atoms_G} \big)
                       : g\in G,\, a\in\Atoms_G \Big\}
     \prefix D_G
  \end{align*}
  and, likewise, $h^{-1} \leftacts D_G \prefix D_G$.
  Hence $h \leftacts D_G = D_G = h^{-1} \leftacts D_G$ for any $h\in H$.
  
  Similarly $D_H := \Join_{a\in\Atoms_H} \Delta_a$ is a balanced element satisfying 
  $g \Leftacts D_H = D_H = g^{-1} \Leftacts D_H$ for any $g\in G$.
  Now define $D := D_G G_H$.
  One has $D = D_G \join D_H = D_G \rightjoin D_H$ by \autoref{ZS-lcm}.
  
  To see that $D$ is balanced, let $g\in G$ and $h\in H$ and consider $g'=h^{-1}\leftacts g$ and $h'=g^{-1}\Leftacts h$.  Using \autoref{action-isomorphism}, \autoref{ZS-lattice-isomorphism}, the invariance of $D_G$ under $h^{-1}\leftacts\cdot$ and the invariance of $D_H$ under $g^{-1}\Leftacts\cdot$, together with the fact that $D_G$ and $D_H$ are balanced, one has
  \begin{align*}
     g\join h\prefix D
       &\Longleftrightarrow (g\prefix D_G
             \;\text{ and }\; h\prefix D_H)    \\
       &\Longleftrightarrow (h^{-1}\leftacts g\prefix h^{-1}\leftacts D_G
             \;\text{ and }\; g^{-1}\Leftacts h\prefix g^{-1}\Leftacts D_H)    \\
       &\Longleftrightarrow (g'\prefix D_G
             \;\text{ and }\; h'\prefix D_H)  \\
       &\Longleftrightarrow (D_G\suffix g'
             \;\text{ and }\; D_H\suffix h')  \\
       &\Longleftrightarrow D \suffix g'\rightjoin h' = g\join h
     \;.
  \end{align*}
  Thus, $D$ is a balanced element of~$K$ whose divisors include the generating set $\Atoms_G\cup\Atoms_H$ of~$K$, so~$D$ is a Garside element for~$K$.
\end{proof}

\begin{remark}
\autoref{E:SemidirectProduct} shows that the construction of a Garside element for the monoid $K=G\zs H$ in the proof of \autoref{ZS-of-Garside-is-Garside} is needed; in general $\Delta_G\Delta_H$ need not be a Garside element for $K$.
\end{remark}

\medskip

\noindent
We finish this section by giving a characterisation of Garside monoids that can be decomposed as a \ZS{} product.

\begin{definition}\label{D:Indecomposable}
 A Garside monoid is \emph{$\zs$-indecomposable}, if it cannot be written as a \ZS{} product of two non-trivial submonoids.
\end{definition}

\begin{remark}
The harder direction of the following \autoref{T:Indecomposable} is implied by~\cite[Proposition~4.5]{Picantin01} and~\autoref{C:ZS-crossed}.
We nevertheless give a direct proof, in order to demonstrate that the result does not depend on the notion of crossed products.
\end{remark}

\begin{theorem}[{cf.\ \cite[Proposition 4.5]{Picantin01}}]\label{T:Indecomposable}
  A Garside monoid $M$ is $\Delta$-pure if and only if it is $\zs$-indecom\-posable.
\end{theorem}
\begin{proof}
We write $\Delta_x$ to mean $\Delta_x^M$ for any $x\in M$.
\bigskip

\noindent
First assume $M=G\zs H$ with non-trivial monoids $G$ and $H$.  Choose two atoms $g\in G$ and $h\in H$.
By \autoref{LocalDeltas}, we have $\id\neq g\prefix \Delta_g\in G$ and $\id\neq h\prefix \Delta_h\in H$.
As $G\cap H=\{\id\}$ by uniqueness of $GH$-decompositions, this implies that $\Delta_g\neq\Delta_h$, so $M$ is not $\Delta$-pure.
\bigskip

\noindent
Now assume that $M$ is not $\Delta$-pure.  By \autoref{P:QuasiCentre}, we can partition the set of atoms of $M$ as $\Atoms = \calG\,\dot\cup\,\calH$, such that $\Delta_g=\Delta_{g'}$ for $g,g'\in\calG$ and $\Delta_g\meet\Delta_h=\id$ for $g\in\calG$, $h\in\calH$.
Let $G := \langle\calG\rangle^+$, $H := \langle\calH\rangle^+$, $D_G := \Delta_g$ for some (hence for all) $g\in\calG$,
and $D_H := \Join_{h\in\calH}\Delta_h = \Delta_{\Join\calH}$.
\medskip

\begin{claim*}
If $a\in\Atoms$, then $a\factor D_G$ iff $a\in\calG$ and $a\factor D_H$ iff $a\in\calH$.
\end{claim*}
If $g\in\calG$ and $h\in\calH$, then one has $g\prefix \Delta_g=D_G$ and $h\prefix \Delta_h\prefix D_H$ by \autoref{P:QuasiCentre}.  Conversely, again using \autoref{P:QuasiCentre}, one has $h\meet D_G\prefix \Delta_h \meet \Delta_g=\id$, so $h\not\prefix D_G$, and finally, $g\prefix D_H$ would imply $\Delta_g\prefix \Join_{h\in\calH}\Delta_h$ by \autoref{P:QuasiCentre}, but the monoid generated by $\{ \Delta_a : a \in\Atoms\}$ is free abelian and $g\in\Atoms\setminus\calH$.
The claim then follows with \autoref{P:QuasiCentre}.
\medskip

\begin{claim*}
One has $D_G\calG = \calG D_G$, $D_G\calH = \calH D_G$, $D_H\calG = \calG D_H$, and $D_H\calH = \calH D_H$.
\end{claim*}
By \autoref{P:QuasiCentre}, the elements $D_G$ and $D_H$ are quasi-central, so one has
$D_G\Atoms = \Atoms D_G$ and $D_H\Atoms = \Atoms D_H$.  Let $g\in\calG$ and $h\in\calH$.
As $x\under h\prefix \Delta_h$ holds for any $x\in M$ by definition of $\Delta_h$,
assuming $D_G g = h D_G$ implies $D_G\under h\prefix g\meet\Delta_h=\id$, so $ h\prefix D_G$, which is a contradiction.
Similarly, $D_H h = g D_H$ would imply $D_H\under g\prefix h\meet\Delta_g=\id$, so $g\prefix D_H$, which is a contradiction.
\medskip

\begin{claim*}
One has $g\under D_H=D_H$ for any $g \in G$, and $h\under D_G=D_G$ for any $h \in H$.
\end{claim*}
First suppose that $g\in \calG$.  Then by the previous claim $g D_H
= D_H g'$ for some atom $g' \in \calG$.
As $g\not\prefix D_H$, we have $g \join D_H = D_H x
\prefix D_H g'$ for some non-trivial~$x$.  Hence $x \prefix g'$ and so
$x = g'$ and the first claim follows in this case.

One has $\id\under D_H=D_H$ by definition.  If $\id\neq g \in G$, then write $g=g_1 g_2 \cdots g_k$ with $g_1,g_2,\ldots g_k\in\calG$.
From the previous case and \autoref{L:Complements} we obtain $(g_1\cdots g_j)\under D_H=g_j\under\big((g_1\cdots g_{j-1})\under D_H\big)=g_j\under D_H=D_H$ for $j=1,\ldots,k$ by induction, hence the first claim is shown.

The second claim is shown in the same way.
\medskip

\begin{claim*}
If $g \prefix D_G$ and $h \prefix D_H$ are non-trivial, then $h\under g$ is a non-trivial prefix of~$D_G$, and $g\under h$ is a non-trivial prefix of~$D_H$.
\end{claim*}
Using \autoref{L:Complements} and the previous claim, we have $g\under h \prefix g\under D_H = D_H$.
If $g\under h$ were trivial then $h\prefix h(h\under g) = g \prefix D_G$, which is a contradiction.

The claim for $h\under g$ is shown in the same way.
\medskip

\begin{claim*}
If $g \prefix D_G$ and $h \prefix D_H$, then $||g||_\Atoms \le ||h\under g||_\Atoms$ and $||h||_\Atoms \le ||g\under h||_\Atoms$.
\end{claim*}
Consider the first inequality.
By \autoref{L:Complements} it is enough to consider the case where $g$ is
an atom and in this case the result follows from the previous claim.

The argument for the second inequality is analogous.
\medskip

\begin{claim*}
For atoms $g \in \calG$ and $h \in \calH$, one has $h \under g \in
\calG$ and $g \under h \in \calH$.
\end{claim*}
\rightskip35mm
Let $g \join D_H = g D_H = D_H g'$ with $g'\in\calG$ and pick $h_1, \ldots, h_k \in\calH$ such that $D_H = h_1 h_2 \cdots h_k$ and $h_1 = h$.  Let $g_1 = g$ and $g_{i+1} = h_i \under
g_i$ for $i=1,\ldots,k$.  By the previous claim we have $||g_1||_\Atoms \le
||g_2||_\Atoms \le \cdots \le ||g_{k+1}||_\Atoms$ and by \autoref{L:Complements} we have $g_{k+1} = g'$.
\hfill
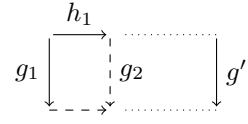
\begin{picture}(0,0)(-10,-14)
\begin{tikzpicture}
 \draw [dashed,->] (0,-1) -- (0.75,-1);
 \draw [dotted] (1,-1) -- (2.2,-1);
 \draw [->] (0.05,0) -- (0.75,0) node [above] at (0.4,0) {$h_1$};
 \draw [dotted] (1,0) -- (2.2,0);
 \draw [->] (2.2,-0.05) -- (2.2,-0.95) node [right] at (2.2,-0.5) {$g'$};
 \draw [dashed,->] (0.8,-0.05) -- (0.8,-0.95) node [right] at (0.8,-0.5) {$g_2$};
 \draw [->] (0,-0.05) -- (0,-0.95) node [left] at (0,-0.5) {$g_1$};
\end{tikzpicture}
\end{picture}
\par
\rightskip0mm

As $g'$ is an atom, we have $||g'||_\Atoms=1$, hence $||g_i||_\Atoms=1$ for
all $i$.  In particular, $g_2 = g \under h$ is an atom.

The argument for $h\under g$ is analogous.
\medskip

\begin{claim*}
The map $g \mapsto h\under g$ for fixed $h\in\calH$ is a bijection on $\calG$
and the map $h \mapsto g\under h$ for fixed $g\in\calG$ is a bijection on $\calH$.
\end{claim*}
Let $h\in\calH$ and assume $h\under g_1=h\under g_2$ for $g_1,g_2\in\calG$.
As $h\prefix D_H$, there exists $\bar h\in H$ such that $h\bar h=D_H$.  Moreover, as $g_1\under h\prefix D_H$ and $g_2\under h\prefix D_H$, there exist $h_1$ and $h_2\in H$ such that $(g_1\under h)h_1=D_H$ respectively $(g_2\under h)h_2=D_H$.  Finally, let $g'_1,g'_2\in\calG$ be such that $g_1 D_H=D_H g'_1$ and
$g_2 D_H=D_H g'_2$.

\hfill\begin{minipage}{0.41\textwidth}
\begin{tikzpicture}[scale=1.2]
 \draw [dashed,->] (0,-1) -- (0.95,-1) node [below] at (0.5,-1) {$g_1\under h$};
 \draw [dashed,->] (1.05,-1) -- (2.95,-1) node [below] at (2,-1) {$h_1$};
 \draw [->] (0.05,0) -- (0.95,0) node [above] at (0.5,0) {$h$};
 \draw [dashed,->] (1.05,0) -- (2.95,0) node [above] at (2,0) {$\bar h$};
 \draw [->] (3,-0.05) -- (3,-0.95) node [right] at (3,-0.5) {$g'_1$};
 \draw [dashed,->] (1,-0.05) -- (1,-0.95) node [right] at (1,-0.5) {$h\under g_1$};
 \draw [->] (0,-0.05) -- (0,-0.95) node [left] at (0,-0.5) {$g_1\phantom{'}$};
\end{tikzpicture}
\end{minipage}
\hfill
\begin{minipage}{0.41\textwidth}
\begin{tikzpicture}[scale=1.2]
 \draw [dashed,->] (0,-1) -- (0.95,-1) node [below] at (0.5,-1) {$g_2\under h$};
 \draw [dashed,->] (1.05,-1) -- (2.95,-1) node [below] at (2,-1) {$h_2$};
 \draw [->] (0.05,0) -- (0.95,0) node [above] at (0.5,0) {$h$};
 \draw [dashed,->] (1.05,0) -- (2.95,0) node [above] at (2,0) {$\bar h$};
 \draw [->] (3,-0.05) -- (3,-0.95) node [right] at (3,-0.5) {$g'_2$};
 \draw [dashed,->] (1,-0.05) -- (1,-0.95) node [right] at (1,-0.5) {$h\under g_2$};
 \draw [->] (0,-0.05) -- (0,-0.95) node [left] at (0,-0.5) {$g_2\phantom{'}$};
\end{tikzpicture}
\end{minipage}
\hfill

As $g_1\not\prefix D_H$, we have $D_H \join g_1 = g_1 D_H = D_H g'_1 = h\big((h\under g_1)\join \bar h\big)$.
Likewise, $D_H \join g_2 = g_2 D_H = D_H g'_2 = h\big((h\under g_2)\join \bar h\big)$, so $h\under g_1=h\under g_2$ implies $g_1=g_2$ (and $g'_1=g'_2$), so the map $g \mapsto h\under g$ is injective.
As $\calG$ is finite, the map is a bijection.

The argument for the map $h \mapsto g\under h$ for fixed $g\in\calG$ is analogous.
\medskip

\begin{claim*}
For $x\in M$, there exist $g_1,g_2\in G$ and $h_1,h_2\in H$ such that one has $g_1h_1=x=h_2g_2$, that is,
$GH$-decompositions and $HG$-decompositions exist.
\end{claim*}
Given $x\in M$, consider any expression for $x$ as a product of atoms of $M$.
By the previous claim, we can move all atoms in either $\calG$ or in $\calH$ to the front of the word, using identities of the form $g(g\under h) = h(h\under g)$ with $g,(h\under g)\in\calG$ and $h,(g\under h)\in\calH$.
\medskip

\begin{claim*}
$GH$-decompositions and $HG$-decompositions are unique.
\end{claim*}
Consider $g\in G$ and $h\in H$ and let $N := ||g||_{\calG}<\infty$.
Since $D_G\calG = \calG D_G$ holds, and as for $a\in\Atoms$ we have $a\prefix D_G$ if and only if $a\in\calG$, one has $g\prefix D_G^N$.
Similarly, $h\prefix D_H^M$ for some~$M$, and thus $D_G^N\meet h\prefix D_G^N\meet D_H^M=\id$.
Writing $D_G^N = g\overline{g}$, we have $g\prefix D_G^N\meet (gh) = g(\overline{g}\meet h) \prefix g(D_G^N\meet h) = g$, so $D_G^N\meet (gh) = g$.
Hence, if $g_1,g_2\in G$ and $h_1,h_2\in H$ are such that $g_1h_1=g_2h_2$, then for $N := \max\{||g_1||_{\calG},||g_2||_{\calG}\}$ one has $g_1 = D_G^N\meet (g_1h_1) = D_G^N\meet (g_2h_2) = g_2$ and, by cancellativity of $M$, then $h_1=h_2$.

Analogously, $D_G^N\rightmeet (hg) = g$ yields the uniqueness of $HG$-decompositions.
\medskip

\noindent
Thus, one has $M=G\zs H$.
\end{proof}


\section{\ZS{} Garside structures} \label{S:Garside-ZS}

We have seen that decomposing a Garside element of a Garside monoid $K=G\zs H$
gives Garside elements for the factors, but that not every pair of Garside elements of the factors can be obtained in this way; cf.\ \autoref{E:SemidirectProduct}.  Clearly one cannot hope to relate the normal form in~$K$ of an element to the normal forms in~$G$ and~$H$ of the factors in its $GH$- and $HG$-decompositions, unless the Garside elements of~$K$,~$G$ and~$H$ are related.
In light of this remark we make the following definition:
\begin{definition}\label{D:ZS-GarsideStructure}
  The tuple $(\Delta_K, \Delta_G, \Delta_H)$ is a \emph{\ZS{} Garside structure} for the \ZS{} product $K=G\zs H$, if:
  \begin{enumerate}[label=(\alph*)]
   \item $K$ is a Garside monoid (and hence~$G$ and~$H$ are also Garside monoids);
   \item $\Delta_K$, $\Delta_G$, $\Delta_H$ are Garside elements for~$K$, $G$, $H$, respectively; and
   \item $\Delta_K = \Delta_G\Delta_H$ holds.
  \end{enumerate}
\end{definition}

\begin{remark}
 The proof of \autoref{ZS-parabolic} shows that in the situation of~\autoref{D:ZS-GarsideStructure},
 one has $ \Delta_G\Delta_H = \Delta_K = \Delta_H\Delta_G$.
\end{remark}

\subsection{Actions and the lattice structures}\label{S:ZappaSzep-Lattice}
A \ZS{} Garside structure allows to describe the lattice structure of the product in terms of the lattice structures of the factors.
By~\cite[Proposition~3.12]{Picantin01}, the lattice of simple elements of the product is the product of the lattices of the simple elements of the factors.
To be able to describe normal forms in \autoref{S:ZappaSzep-NF}, we need to analyse how the actions of the factors on each other interact with the lattice structures; this is the content of this section.

\begin{theorem}
  Suppose $(\Delta_K, \Delta_G, \Delta_H)$ is a \ZS{} Garside structure for the \ZS{} product $K=G\zs H$.
  For all $g \in G$ and $h\in H$, one has
  \begin{align*}
    {\inf}_K(g \join h) &= \min\big({\inf}_G(g), {\inf}_H(h)\big) \\
    {\sup}_K(g \join h) &= \max\big({\sup}_G(g), {\sup}_H(h)\big) \;,
  \end{align*}
  where ${\inf}_K$, ${\inf}_G$, ${\inf}_H$ and ${\sup}_K$, ${\sup}_G$, ${\sup}_H$ denote the infimum respectively the supremum with respect to the Garside structures of $K$, $G$ and $H$ given by $\Delta_K$, $\Delta_G$ and $\Delta_H$, respectively.
\end{theorem}
\begin{proof}
  The number ${\inf}_K(g \join h)$ is the largest integer $\ell$ such that $\Delta_K^\ell
  \prefix g \join h$.  By \autoref{ZS-lattice-isomorphism}, this is
  equal to the largest integer $\ell$ such that $\Delta_G^\ell \prefix g$ and
  $\Delta_H^\ell \prefix h$, which is the minimum of ${\inf}_G(g)$ and ${\inf}_H(h)$.

  Similarly, ${\sup}_K(g \join h)$ is the smallest integer $\ell$ such that
  $g \join h \prefix \Delta_K^\ell$.  This is equal to the smallest integer $\ell$
  such that $g \prefix \Delta_G^\ell$ and $h \prefix \Delta_H^\ell$, which
  is the maximum of ${\sup}_G(g)$ and ${\sup}_H(h)$.
\end{proof}

\begin{lemma} \label{actions-preserve-Delta}
  Suppose $(\Delta_K, \Delta_G, \Delta_H)$ is a \ZS{} Garside structure for the \ZS{} product $K=G\zs H$.
  For all $g \in G$, $h \in H$, the following identities hold:
  \begin{align*}
    h \leftacts \Delta_G &= \Delta_G &
    \Delta_H \rightacts g &= \Delta_H \\
    g \Leftacts \Delta_H &= \Delta_H &
    \Delta_G \Rightacts h &= \Delta_G 
  \end{align*}
\end{lemma}
\begin{proof}
  As $\leftacts$ is an action, we can assume that $h$ is a simple
  element.  Suppose $h \Delta_G = g' h'$, i.e.\ $g' = h \leftacts
  \Delta_G$.  Observe that $(\rightcomplement_H h) h \Delta_G =
  \Delta_K$, so $\Delta_K \suffix g' h'$ and, in particular, $g'$ is a
  simple element.  Also, $h \join \Delta_G = h (h^{-1}
  \leftacts\Delta_G)\in\Simples$ and thus $h^{-1} \leftacts
  \Delta_G\in\Simples$.
  For $x = (\Delta_G^{-1} \Leftacts h) \complement_G (h^{-1} \leftacts
  \Delta_G)$ we have
  \begin{align*}
    \Delta_G x &= \Delta_G (\Delta_G^{-1} \Leftacts h) 
                  \complement_G (h^{-1} \leftacts \Delta_G) \\
               &= h (\Delta_G \Rightacts (\Delta_G^{-1} \Leftacts h))
                  \complement_G (h^{-1} \leftacts \Delta_G) 
                   \\
               &= h (h^{-1} \leftacts \Delta_G)
                  \complement_G (h^{-1} \leftacts \Delta_G)
                  &&\text{by \autoref{action-inverse}} \\
               &= h \Delta_G \;.
  \end{align*}
  Hence $\Delta_G \prefix h \Delta_G = g' h'$ and so $\Delta_G \prefix
  g'$ which implies that $g' = \Delta_G$.
  The other claims are shown in the same way.
\end{proof}

\begin{corollary}\label{ActionsMapSimplesToSimples}
 Suppose $(\Delta_K, \Delta_G, \Delta_H)$ is a \ZS{} Garside structure for the \ZS{} product $K=G\zs H$.
 \smallskip
 
 \noindent
 For all $h \in H$, one has
 \begin{align*}
  g \in \Simples_G
    &\Longleftrightarrow h \leftacts g \in \Simples_G
     \Longleftrightarrow h^{-1} \leftacts g \in \Simples_G \\
    &\Longleftrightarrow g \Rightacts h \in \Simples_G
     \Longleftrightarrow g \Rightacts h^{-1} \in \Simples_G
     \;.
 \end{align*}
 For all $g \in G$, one has
 \begin{align*}
  h \in \Simples_H
    &\Longleftrightarrow h \rightacts g \in \Simples_H
     \Longleftrightarrow h \rightacts g^{-1} \in \Simples_H \\
    &\Longleftrightarrow g \Leftacts h \in \Simples_H
     \Longleftrightarrow g^{-1} \Leftacts h \in \Simples_H
     \;.
 \end{align*}
\end{corollary}
\begin{proof}
 The claim follows from \autoref{actions-preserve-Delta} with \autoref{ZS-actions} and \autoref{action-inverse-2}.
\end{proof}

\begin{lemma} \label{complement-action}
  Suppose $(\Delta_K, \Delta_G, \Delta_H)$ is a \ZS{} Garside structure for the \ZS{} product $K=G\zs H$.
  \smallskip
 
  \noindent
  For all $g \in \Simples_G$ and $h \in H$, one has
  \begin{align*}
    \complement_G (h\leftacts g) = (h\rightacts g) \leftacts \complement_G g \qquad\text{and}\qquad
    \complement_G (g\Rightacts h) = h^{-1} \leftacts \complement_G g \;.
   \end{align*}
  For all $g \in G$ and $h \in \Simples_H$, one has
  \begin{align*}
    \complement_H (g\Leftacts h) = (g\Rightacts h) \Leftacts \complement_H h \qquad\text{and}\qquad
    \complement_H (h\rightacts g) = g^{-1} \Leftacts \complement_H h \;.
   \end{align*}
\end{lemma}
\begin{proof}
  Consider the following.
  \begin{align*}
    (h \leftacts g) \left( (h \rightacts g) \leftacts \complement_G g \right)
      & = h \leftacts (g \complement_G g) 
         &&\text{by \autoref{ZS-actions}} \\
      & = h \leftacts \Delta_G \\
      & = \Delta_G &&\text{by \autoref{actions-preserve-Delta}}
  \end{align*}
  Hence $\complement_G (h \leftacts g) = (h \rightacts g) \leftacts
  \complement_G g$. 

  For the right action we have the following.
  \begin{align*}
    (h &\leftacts g) \Leftacts \big( 
      (h \rightacts g)(g^{-1} \Leftacts \complement_H h)
    \big) \\
    &= \big(
      (h \leftacts g) \Leftacts (h \rightacts g)
    \big) \Big(\big(
      (h \leftacts g) \Rightacts (h \rightacts g) 
    \big) \Leftacts (g^{-1} \Leftacts \complement_H h)\Big)
         &&\!\text{by \autoref{ZS-actions}} \\
    &= h \big(g \Leftacts (g^{-1} \Leftacts \complement_H h)\big)
         &&\!\text{by \autoref{left-right-identity}} \\
    &= h \complement_H h = \Delta_H
  \end{align*}
  So, by \autoref{actions-preserve-Delta}, $(h \rightacts g)(g^{-1}
  \Leftacts \complement_H h) = \Delta_H$, i.e.\ \mbox{$\complement_H (h
  \rightacts g) = (g^{-1} \Leftacts \complement_H h)$}.
  The remaining identities follow with \autoref{transforming-expressions}.
\end{proof}

\begin{lemma}\label{L:ZS-complement}
  Suppose $(\Delta_K, \Delta_G, \Delta_H)$ is a \ZS{} Garside structure for the \ZS{} product $K=G\zs H$.
  Then, for all $g \in \Simples_G$ and $h \in \Simples_H$, one has
  \[ \complement_K (g \join h) 
     = \complement_G (h^{-1} \leftacts g) \join
       \complement_H (g^{-1} \Leftacts h)\;. \]
\end{lemma}
\begin{proof}
  Suppose that $g \join h = gh' = hg'$.  Now, using \autoref{actions-preserve-Delta}, we have
  \begin{align*}
    h g' \complement_G g' (\Delta_G^{-1} \Leftacts \complement_H h)
    & = h \Delta_G (\Delta_G^{-1} \Leftacts \complement_H h)
      = h \complement_H h \Delta_G = \Delta_H \Delta_G = \Delta_K\;, \\
    \intertext{and, similarly,}
    g h' \complement_H h' (\Delta_H^{-1} \leftacts \complement_G g)
    & = g \Delta_H (\Delta_H^{-1} \leftacts \complement_G g)
      = g \complement_G g \Delta_H = \Delta_G \Delta_H = \Delta_K \;.
  \end{align*}
  Hence
  \begin{align*}
    \complement_K (g \join h)
    &= \complement_G g' (\Delta_G^{-1} \Leftacts \complement_H h)
     = \complement_H h' (\Delta_H^{-1} \leftacts \complement_G g)\;, \\
    \intertext{And thus,}
    \complement_K (g \join h)
    & = \complement_G g' \join \complement_H h'
      = \complement_G (h^{-1}\leftacts g) \join \complement_H (g^{-1} \Leftacts h)\;.
  \end{align*}
\end{proof}

\subsection{Normal forms and \ZS{} Garside structures} \label{S:ZappaSzep-NF}

We will show in this section that, with respect to the Garside elements in a \ZS{} Garside structure for $K=G\zs H$, the language of normal form words in the product $K$ can be described in terms of the Cartesian product of the languages of normal form words in the factors $G$ and $H$.

Recall that we write $x_1|x_2|\cdots|x_\ell$ to denote a word in (non-trivial) simple elements together with the proposition that this word is in left normal form, that we write $\calLbar$ for the language of words in normal form and $\calL$ for the language given by restricting the alphabet to the set of proper simple elements.

\begin{definition}
  The set of equations \eqref{action-on-products}, from
  \autoref{ZS-actions}, gives us a natural way to extend the actions
  on elements to actions on strings of elements.  We can define the
  actions recursively as follows: The actions take the empty string to
  the empty string, act on length one strings by acting on the
  element, and act on longer strings by
  \begin{align*}
    h \leftacts (g_1 \ldot g_2 \ldot \cdots \ldot g_\ell) &= 
      (h \leftacts g_1) \ldot \big(
        (h \rightacts g_1) \leftacts (g_2 \ldot g_3 \ldot \cdots \ldot g_\ell)
      \big) \;, \\
    (h_1 \ldot h_2 \ldot \cdots \ldot h_\ell) \rightacts g &= 
      \big(
        (h_1 \ldot h_2 \ldot \cdots \ldot h_{\ell-1}) \rightacts (h_\ell \leftacts g)
      \big) \ldot (h_\ell \rightacts g) \;, \\
    g \Leftacts (h_1 \ldot h_2 \ldot \cdots \ldot h_\ell) &=
      (g \Leftacts h_1) \ldot \big(
        (g \Rightacts h_1) \Leftacts (h_2 \ldot h_3 \ldot \cdots \ldot h_\ell)
      \big) \;, \\
    (g_1 \ldot g_2 \ldot \cdots \ldot g_\ell) \Rightacts h &= 
      \big(
        (g_1 \ldot g_2 \ldot \cdots \ldot g_{\ell-1}) \Rightacts (g_\ell \Leftacts h)
      \big) \ldot (g_\ell \Rightacts h) \;.
  \end{align*}
  Likewise, if $G$ and $H$ act on each other by bijections, we can extend the inverse actions to strings by
   \begin{align*}
    h^{-1} \leftacts (g_1 \ldot g_2 \ldot \cdots \ldot g_\ell) &= 
      (h^{-1} \leftacts g_1) \ldot \big(
        (g_1^{-1} \Leftacts h)^{-1} \leftacts (g_2 \ldot g_3 \ldot \cdots \ldot g_\ell)
      \big) \;, \\
    (h_1 \ldot h_2 \ldot \cdots \ldot h_\ell) \rightacts g ^{-1}&= 
      \big(
        (h_1 \ldot h_2 \ldot \cdots \ldot h_{\ell-1}) \rightacts (g \Rightacts h_\ell^{-1})^{-1}
      \big) \ldot (h_\ell \rightacts g^{-1}) \;, \\
    g^{-1} \Leftacts (h_1 \ldot h_2 \ldot \cdots \ldot h_\ell) &=
      (g^{-1} \Leftacts h_1) \ldot \big(
        (h_1^{-1} \leftacts g)^{-1} \Leftacts (h_2 \ldot h_3 \ldot \cdots \ldot h_\ell)
      \big) \;, \\
    (g_1 \ldot g_2 \ldot \cdots \ldot g_\ell) \Rightacts h^{-1} &= 
      \big(
        (g_1 \ldot g_2 \ldot \cdots \ldot g_{\ell-1}) \Rightacts (h \rightacts g_\ell^{-1})^{-1}
      \big) \ldot (g_\ell \Rightacts h^{-1}) \;.
  \end{align*}
 
  By \autoref{ZS-actions} and \autoref{action-inverse-2}, these actions on strings of elements commute
  with the multiplication map $g_1 \ldot g_2 \ldot \cdots \ldot g_\ell
  \mapsto g_1 g_2 \cdots g_\ell$.
\end{definition}

\begin{proposition}\label{P:ZS-NF}
  Suppose $(\Delta_K, \Delta_G, \Delta_H)$ is a \ZS{} Garside structure for the \ZS{} product $K=G\zs H$, and let
  $g_1, g_2 \in \Simples_G$ and $h_1, h_2 \in \Simples_H$.
  
  Then one has
  $\complement_K(g_1 \join h_1) \meet (g_2 \join h_2) = \id$ if and
  only if $\complement_G (h_1^{-1} \leftacts g_1) \meet g_2 = \id$ and
  $\complement_H (g_1^{-1} \Leftacts h_1) \meet h_2 = \id$.
  
  If, moreover, $g_2\neq\id$ and $h_2\neq\id$, then one has
  $(g_1 \join h_1)
  | (g_2 \join h_2)$ if and only if $(h_1^{-1} \leftacts g_1) | g_2$
  and $(g_1^{-1} \Leftacts h_1) | h_2$.
\end{proposition}
\begin{proof}
  By \autoref{ZS-lattice-isomorphism} and \autoref{L:ZS-complement}, we have 
  \[  \complement_K (g_1 \join h_1) \meet (g_2 \join h_2)
      = (\complement_G (h_1^{-1} \leftacts g_1) \meet g_2) \join
        (\complement_H (g_1^{-1} \Leftacts h_1) \meet h_2) \;, \]
  so $\complement_K(g_1 \join h_1) \meet (g_2 \join h_2) = \id$ if and
  only if $\complement_G (h_1^{-1} \leftacts g_1) \meet g_2 = \id$ and
  $\complement_H (g_1^{-1} \Leftacts h_1) \meet h_2 = \id$, so the first claim holds.
  
  The second claim follows, as for simple elements $s_1,s_2$ of any Garside monoid, one has $s_1|s_2$ if and only if $\complement s_1\meet s_2=\id$ and $s_2\neq\id$ by definition.
\end{proof}

\begin{corollary} \label{normal-form-criteria}
  Suppose $(\Delta_K, \Delta_G, \Delta_H)$ is a \ZS{} Garside structure for the \ZS{} product $K=G\zs H$, and let
  $g_1, g_2 \in \Simples_G$ and $h_1, h_2 \in \Simples_H$.
  
  Then the following hold:
    \begin{align*}
       \complement_K (g_1 h_1) \meet (g_2 h_2) = \id
         &\quad\Longleftrightarrow \quad
         \left\{ \begin{array}{r@{}c@{}l}
                     \complement_G (g_1 \Rightacts h_1) &\meet& g_2 = \id \\
                      &\text{and} \\
                     \complement_H h_1 &\meet& (g_2 \Leftacts h_2) = \id
                 \end{array} \right. \\[2ex]
      \complement_K (g_1 h_1) \meet (h_2 g_2) = \id
         &\quad\Longleftrightarrow \quad
         \left\{ \begin{array}{r@{}c@{}l}
                     \complement_G (g_1 \Rightacts h_1) &\meet& (h_2 \leftacts g_2) = \id \\
                      &\text{and} \\
                     \complement_H h_1 &\meet& h_2 = \id
                 \end{array} \right.
    \end{align*}
    \begin{align*}
       \complement_K (h_1 g_1) \meet (g_2 h_2) = \id
         &\quad\Longleftrightarrow \quad
         \left\{ \begin{array}{r@{}c@{}l}
                     \complement_G g_1 &\meet& g_2 = \id \\
                      &\text{and} \\
                     \complement_H (h_1 \rightacts g_1) &\meet& (g_2 \Leftacts h_2) = \id
                 \end{array} \right. \\[2ex]
       \complement_K (h_1 g_1) \meet (h_2 g_2) = \id
         &\quad\Longleftrightarrow \quad
         \left\{ \begin{array}{r@{}c@{}l}
                     \complement_G g_1 &\meet& (h_2 \leftacts g_2) = \id \\
                      &\text{and} \\
                     \complement_H (h_1 \rightacts g_1) &\meet& h_2 = \id
                 \end{array} \right.
    \end{align*}

  If, moreover, $g_2\neq\id$ and $h_2\neq\id$, then one has the following:
  \[ \begin{array}{c@{\quad \Longleftrightarrow \quad}r@{}l@{$\quad$and$\quad$}r@{}l}
    g_1 h_1 | g_2 h_2
      & (g_1 \Rightacts h_1) &| g_2
      & h_1 &| (g_2 \Leftacts h_2) \\
    g_1 h_1 | h_2 g_2 
      & (g_1 \Rightacts h_1)
      &| (h_2 \leftacts g_2) & h_1 &| h_2 \\
    h_1 g_1 | g_2 h_2 
      & g_1 &| g_2 & (h_1 \rightacts g_1)
      &| (g_2 \Leftacts h_2) \\
    h_1 g_1 | h_2 g_2
      & g_1 &| (h_2 \leftacts g_2)
      & (h_1 \rightacts g_1) &| h_2
  \end{array} \]
\end{corollary}
\begin{proof}
  The equivalences in the first list follow by \autoref{left-right-identity} and \autoref{P:ZS-NF} together with the fact that,
  for all $g \in G$ and $h \in H$, one has $g h = g \join (g \Leftacts h)$ and $h g = h \join (h \leftacts g)$.
  The equivalences in the second list then follow with \autoref{action-on-id}.
\end{proof}

\begin{remark}
\autoref{P:ZS-NF} and \autoref{normal-form-criteria} provide explicit translations, in both directions, between a deterministic finite state automaton accepting the regular language of normal form words over the alphabet~$\Simples_K^*$ and a pair of deterministic finite state automata accepting the regular languages of normal form words over the alphabets~$\Simples_G^*$ and~$\Simples_H^*$, respectively.
\end{remark}

\begin{proposition}\label{AlgorithmDecomposition}
  Suppose $(\Delta_K, \Delta_G, \Delta_H)$ is a \ZS{} Garside structure for the \ZS{} product $K=G\zs H$.
  \smallskip

  \noindent
  Given the normal form $g_1 h_1 | \cdots | g_m h_m \in \calLbar_K$ of $k\in K$ with $GH$-decomposition $k=gh$, the following algorithm computes the normal forms $\mathsf{Word}_G \in \calLbar_G$ of~$g$ and $\mathsf{Word}_H \in \calLbar_H$ of~$h$.
  
  \begin{algorithmic}[1]
    \STATE $\mathsf{Word}_H \leftarrow g_1 h_1 | g_2 h_2 | \cdots | g_m h_m$
    \STATE $\mathsf{Word}_G \leftarrow \epsilon$
    \REPEAT
    \STATE Write each simple factor of $\mathsf{Word}_H$ as a
      $GH$-decomposition, i.e.\\ $\qquad\mathsf{Word}_H = g'_1 h'_1 |
      g'_2 h'_2 | \cdots | g'_\ell h'_\ell$
    \IF{$g'_1 \ne \id$}
    \STATE $\mathsf{Word}_G \leftarrow \mathsf{Word}_G \ldot g'_1$
    \IF{$h'_\ell \ne \id$}
    \STATE $\mathsf{Word}_H \leftarrow h'_1 g'_2 | h'_2 g'_3 | \cdots | h'_{\ell-1} g'_\ell | h'_\ell$
    \ELSE
    \STATE $\mathsf{Word}_H \leftarrow h'_1 g'_2 | h'_2 g'_3 | \cdots | h'_{\ell-1} g'_\ell$
    \ENDIF
    \ENDIF
    \UNTIL{$g'_1 = \id$}
    \RETURN $(\mathsf{Word}_G, \mathsf{Word}_H)$
  \end{algorithmic}
\end{proposition}
\begin{proof}
By \autoref{normal-form-criteria}, the returned words are in normal form.
\end{proof}

\begin{proposition} \label{action-preserves-NF}
  Suppose $(\Delta_K, \Delta_G, \Delta_H)$ is a \ZS{} Garside structure for the \ZS{} product $K=G\zs H$, that
  $g_1 \ldot g_2 \ldot \cdots \ldot g_\ell$ is a word in $\Simples_G^*$ and that $h \in H$.
  Define \[ g'_1 \ldot g'_2 \ldot \cdots \ldot g'_\ell := h \leftacts (g_1 \ldot g_2 \ldot \cdots
  \ldot g_\ell) \;. \]
  
  For $i=1,\ldots\ell-1$, one has $\complement_G g_i \meet g_{i+1} = \id$ if and only if
  $\complement_G g'_i \meet g'_{i+1} = \id$.  In other words,
  $g_1 | g_2 | \cdots | g_\ell$ if and only if $g'_1 | g'_2 | \cdots | g'_\ell$.
\end{proposition}
\begin{proof}
  First observe that, for $i=1,\ldots,\ell$, we have $g'_i=\id$ if and only if $g_i=\id$ by
  \autoref{action-on-id}.
  
  Now consider the case when $\ell = 2$. We have:
  \begin{align*}
    \complement_G g_1' \meet g_2' 
      &= \complement_G(h \leftacts g_1) \meet 
         \left((h \rightacts g_1) \leftacts g_2 \right) \\
      &= \left((h \rightacts g_1) \leftacts \complement_G g_1\right) \meet 
         \left((h \rightacts g_1) \leftacts g_2 \right)
         &&\text{by \autoref{complement-action}} \\
      &= (h \rightacts g_1) \leftacts \left(\complement_G g_1 \meet g_2 \right)
         &&\text{by \autoref{action-isomorphism}} 
  \end{align*}
  Hence, by \autoref{action-on-id}, $\complement_G g'_1 \meet g'_2 =
  \id$ if and only if $\complement_G g_1 \meet g_2 = \id$.  As $g'_2=\id$ if and only if $g_2=\id$, we have
  $g'_1 | g'_2$ if and only if $g_1 | g_2$ as desired.

  For the general case, if we let $h'_i = h \rightacts g_1 g_2 \cdots
  g_{i-1}$ then we have that $g'_i \ldot g'_{i+1} = h'_i \leftacts
  (g_i. g_{i+1})$.  So each length $2$ subword reduces to the $\ell = 2$
  case.
\end{proof}

\begin{corollary}
  Suppose $(\Delta_K, \Delta_G, \Delta_H)$ is a \ZS{} Garside structure for the \ZS{} product $K=G\zs H$.

  The actions on words fix setwise the languages $\calL_G$ and $\calL_H$.
\end{corollary}
\begin{proof}
  This follows from \autoref{action-preserves-NF} as, by \autoref{actions-preserve-Delta}, the initial
  power of~$\Delta$ in a word in normal form must be preserved by the
  actions.
\end{proof}

\begin{lemma} \label{push-normal-form}
  Suppose $(\Delta_K, \Delta_G, \Delta_H)$ is a \ZS{} Garside structure for the \ZS{} product $K=G\zs H$.
  
  If $g_1, g_2, \ldots, g_\ell \in \Simples_G$ and
  $h, h_1, h_2, \ldots, h_\ell \in \Simples_H$ with $\complement_H h \meet (g_1 \Leftacts h_1) = \id$ and
  $g_1 h_1 | g_2 h_2 | \cdots | g_\ell h_\ell$ hold,
  then one has $h g_1 | h_1 g_2 | \cdots | h_{\ell-1} g_\ell$ and, moreover,
  $h g_1 | h_1 g_2 | \cdots | h_{\ell-1} g_\ell | h_\ell$ if $h_\ell\neq\id$.
\end{lemma}
\begin{proof}
  If we let $h_0 = h$ and $g_{\ell+1} = \id$ then, by \autoref{normal-form-criteria}, the
  hypotheses imply
  \[
    \forall i \in \{1, \ldots, \ell\}, \qquad 
    \complement_G (g_i \Rightacts h_i) \meet g_{i+1} = \id
    \quad \text{and} \quad
    \complement_H h_{i-1} \meet (g_i \Leftacts h_i) = \id
    \;.
  \]
  Moreover, either $g_i\neq\id$ for $i=1,\ldots,\ell$, or $h_i\neq\id$ for $i=0,1,\ldots,\ell$.
  
  Now consider the following.
  \begin{align*}
    (g_i \Leftacts {}&h_i) \leftacts \big((g_i \Rightacts h_i) \ldot g_{i+1}\big) \\
    &= \big( (g_i \Leftacts h_i) \leftacts (g_i \Rightacts h_i) \big) .
       \big( (g_i \Leftacts h_i) \rightacts (g_i \Rightacts h_i) \big) 
         \leftacts g_{i+1} \\
    &= g_i \ldot (h_i \leftacts g_{i+1}) && \text{by \autoref{left-right-identity}}
  \end{align*}
  \begin{align*}
    \big(h_{i-1} \ldot (&g_i \Leftacts h_i)\big) \rightacts (g_i \Rightacts h_i) \\
    &= h_{i-1} \rightacts 
         \big((g_i \Leftacts h_i) \leftacts (g_i \Rightacts h_i)\big) .
       (g_i \Leftacts h_i) \rightacts (g_i \Rightacts h_i) \\
    &= (h_{i-1} \rightacts g_i) \ldot h_i && \text{by \autoref{left-right-identity}}
  \end{align*}
  So, by \autoref{action-preserves-NF}, we have for $i=1,\ldots,\ell-1$ that
  $\complement_G g_i \meet (h_i \leftacts g_{i+1}) = \id$ and
  $\complement_H (h_{i-1} \rightacts g_i) \meet h_i = \id$, which, using
  \autoref{normal-form-criteria}, implies the claim.
\end{proof}

\begin{proposition}\label{AlgorithmNF}
 Suppose $(\Delta_K, \Delta_G, \Delta_H)$ is a \ZS{} Garside structure for the \ZS{} product $K=G\zs H$.
  
 Given $g_1|\cdots|g_m \in \calLbar_G$ and $h_1|\cdots|h_n \in \calLbar_H$, the following algorithm computes the normal form of $g_1 \cdots g_m h_1 \cdots h_n$.

  \begin{algorithmic}[1]
    \STATE $\mathsf{Word}_G \leftarrow g_1 | g_2 | \cdots | g_m$
    \STATE $\mathsf{Word}_K \leftarrow h_1 | h_2 | \cdots | h_n$
    \WHILE{$\mathsf{Word}_G \ne \epsilon$} 
      \STATE $\mathsf{Word}_G \ldot g \leftarrow \mathsf{Word}_G$
             \quad /* extract last simple factor of normal form */
      \STATE Write each simple factor of $\mathsf{Word}_K$ as a
        $HG$-decomposition, i.e.\\ $\qquad\mathsf{Word}_K = h'_1 g'_1 | h'_2
        g'_2 | \cdots | h'_\ell g'_\ell$.
      \IF{$g'_\ell = \id$}
        \STATE \label{rewrite-1} $\mathsf{Word}_K \leftarrow g h'_1 | g'_1
          h'_2 | \cdots | g'_{\ell-1} h'_\ell$
      \ELSE
        \STATE \label{rewrite-2} $\mathsf{Word}_K \leftarrow g h'_1 | g'_1
          h'_2 | \cdots | g'_{\ell-1} h'_\ell | g'_\ell$
      \ENDIF
    \ENDWHILE
    \RETURN $\mathsf{Word}_K$
  \end{algorithmic}
\end{proposition}
\begin{proof}
By \autoref{push-normal-form}, in each iteration the word computed in line \ref{rewrite-1}, or \ref{rewrite-2}, is in normal form.
\end{proof}

\begin{remark}
 \autoref{AlgorithmDecomposition} and \autoref{AlgorithmNF} provide explicit and effective translations, in both directions, between normal forms in an internal \ZS{} product and the normal forms in its factors.
 
 As the existence of effectively computable normal forms is the main motivation for Garside theory, these explicit translations are some of the key results of this paper:  \autoref{AlgorithmDecomposition} and \autoref{AlgorithmNF} make it possible to reduce computational questions according to a decomposition of a Garside monoid as a product of simpler constituents.
 
 \autoref{AlgorithmDecomposition} and \autoref{AlgorithmNF} are essentially the constructive versions of \autoref{NF-Bijection-1} and \autoref{NF-Bijection-2}.
\end{remark}

\begin{theorem}\label{NF-Bijection-1}
  Suppose $(\Delta_K, \Delta_G, \Delta_H)$ is a \ZS{} Garside structure for the \ZS{} product $K=G\zs H$.
  \smallskip

  \noindent
  Then the map
  $\phi \from \calLbar_G \times \calLbar_H \to \calLbar_K$ given by
  \[
     \phi\big(g_1 | g_2 | \cdots | g_m \,,\, h_1 | h_2 | \cdots | h_n\big)
       = \NF\big(g_1 g_2 \cdots g_m h_1 h_2 \cdots h_n\big)
  \]
  is a bijection.
\end{theorem}
\begin{proof}
Clearly $\NF\big(g_1 g_2 \cdots g_m h_1 h_2 \cdots h_n\big) \in \calLbar_K$, so the map $\phi$ is well-defined.

As each $k_1|\cdots|k_\ell\in\calLbar_K$ is the image of $\big(\NF(g)\,,\,\NF(h)\big)$, where $g h$ is the $GH$-decomposition of $k_1\cdots k_\ell \in K$, the map $\phi$ is surjective.

Now assume $g_1|\cdots|g_m$ and $g'_1|\cdots|g'_p$ in~$\calLbar_G$ and $h_1|\cdots|h_n$ and $h'_1|\cdots|h'_q$ in~$\calLbar_H$ satisfy $\phi\big(g_1 | \cdots | g_m \,,\, h_1 | \cdots | h_n\big)
= \phi\big(g'_1 | \cdots | g'_p \,,\, h'_1 | \cdots | h'_q\big)$.
Then one has $g_1 \cdots g_m  h_1 \cdots h_n = g'_1 \cdots g'_p  h'_1 \cdots h'_q$ and thus, by uniqueness of $GH$-de\-com\-po\-si\-tions, $g_1 \cdots g_m  = g'_1 \cdots g'_p$ and $h_1 \cdots h_n = h'_1 \cdots h'_q$.
Uniqueness of normal forms then yields $g_1|\cdots|g_m = g'_1|\cdots|g'_p$ and $h_1|\cdots|h_n = h'_1|\cdots|h'_q$, so the map~$\phi$ is injective.
\end{proof}

\begin{remark}
The map $\phi$ can be realised using the algorithm of \autoref{AlgorithmNF} and its inverse, $\phi^{-1}$, can be realised using the algorithm from \autoref{AlgorithmDecomposition}.
\end{remark}

\begin{corollary}\label{NF-Bijection-2}
  Suppose $(\Delta_K, \Delta_G, \Delta_H)$ is a \ZS{} Garside structure for the \ZS{} product $K=G\zs H$.
  \smallskip

  \noindent
  Then the map
  $\psi \from \calLbar_G \times \calLbar_H \to \calLbar_K$ given by
  \[
     \psi\big(g_1 | g_2 | \cdots | g_m \,,\, h_1 | h_2 | \cdots | h_n\big)
       = \NF\big( (g_1 g_2 \cdots g_m) \join (h_1 h_2 \cdots h_n) \big)
  \]
  is a bijection.
\end{corollary}
\begin{proof}
  By \autoref{action-preserves-NF}, $h'_1 \ldot h'_2 \ldot \cdots \ldot h'_n =
  (g_1 g_2 \cdots g_m)^{-1} \Leftacts (h_1 \ldot h_2 \ldot \cdots
  \ldot h_n)$ is a word in normal form. So, by \autoref{ZS-lcm}, $\psi(g_1
  \ldot g_2 \ldot \cdots \ldot g_m, h_1 \ldot h_2 \ldot \cdots \ldot h_n) = \phi(g_1 \ldot g_2
  \ldot \cdots \ldot g_m, (g_1 g_2 \cdots g_m)^{-1} \Leftacts (h_1 \ldot h_2
  \ldot \cdots \ldot h_n))$.  Therefore, as $\psi$ is a composition of
  bijections, it is a bijection.
\end{proof}

\bibliographystyle{alpha-sjt}
\bibliography{bibliography}

\providecommand{\etalchar}[1]{$^{#1}$}
\providecommand{\MR}{\relax\ifhmode\unskip\space\fi MR}
\providecommand{\arXiv}{\relax\ifhmode\unskip\space\fi arXiv:}
\begin{thebibliography}{BRRW14}

\bibitem[ACIM09]{Agore09}
A.~L. Agore, A.~Chirv{\u{a}}situ, B.~Ion, and G.~Militaru.
\newblock Bicrossed products for finite groups.
\newblock {\em Algebr. Represent. Theory}, 12(2-5):481--488, October 2009.
  \MR{2501197 (2010e:20050)}

\bibitem[AM11]{AgoreMilitaru11}
A.~L. Agore and G.~Militaru.
\newblock Extending structures {II}: {T}he quantum version.
\newblock {\em J. Algebra}, 336:321--341, 2011. \MR{2802546}

\bibitem[Bri05]{Brin05}
Matthew~G. Brin.
\newblock On the {Z}appa-{S}z\'ep product.
\newblock {\em Comm. Algebra}, 33(2):393--424, 2005. \MR{2124335 (2005k:20170)}

\bibitem[Bri07]{Brin07}
Matthew~G. Brin.
\newblock The algebra of strand splitting. {I}. {A} braided version of
  {T}hompson's group {$V$}.
\newblock {\em J. Group Theory}, 10(6):757--788, 2007. \MR{2364825
  (2009g:20092)}

\bibitem[BRRW14]{Brownlowe}
Nathan Brownlowe, Jacqui Ramagge, David Robertson, and Michael~F. Whittaker.
\newblock Zappa-{S}z\'ep products of semigroups and their {$C^\ast$}-algebras.
\newblock {\em J. Funct. Anal.}, 266(6):3937--3967, 2014. \MR{3165249}

\bibitem[Cas41]{Casadio41}
Giuseppina Casadio.
\newblock Costruzione di gruppi come prodotto di sottogruppi permutabili.
\newblock {\em Univ. Roma e Ist. Naz. Alta Mat. Rend. Mat. e Appl. (5)},
  2:348--360, 1941. \MR{0018176 (8,251d)}

\bibitem[DDG{\etalchar{+}}15]{GarsideBook}
Patrick Dehornoy, François Digne, Eddy Godelle, Daan Krammer, and Jean Michel.
\newblock {\em Foundations of {G}arside theory}, volume~22 of {\em EMS Tracts
  in Mathematics}.
\newblock European Mathematical Society (EMS), Z\"urich, 2015. \MR{3362691}

\bibitem[Deh02]{Dehornoy02}
Patrick Dehornoy.
\newblock Groupes de {G}arside.
\newblock {\em Ann. Sci. \'Ecole Norm. Sup. (4)}, 35(2):267--306, 2002.
  \MR{1914933 (2003f:20068)}

\bibitem[DP99]{DehornoyParis99}
Patrick Dehornoy and Luis Paris.
\newblock Gaussian groups and {G}arside groups, two generalisations of {A}rtin
  groups.
\newblock {\em Proc. London Math. Soc. (3)}, 79(3):569--604, 1999. \MR{1710165
  (2001f:20061)}

\bibitem[God07]{Godelle07}
Eddy Godelle.
\newblock Parabolic subgroups of {G}arside groups.
\newblock {\em J. Algebra}, 317(1):1--16, 2007. \MR{2360138 (2008h:20054)}

\bibitem[God10]{Godelle10}
Eddy Godelle.
\newblock Parabolic subgroups of {G}arside groups {II}: ribbons.
\newblock {\em J. Pure Appl. Algebra}, 214(11):2044--2062, 2010. \MR{2645337
  (2011j:20091)}

\bibitem[Kun83]{Kunze83}
M.~Kunze.
\newblock Zappa products.
\newblock {\em Acta Math. Hungar.}, 41(3-4):225--239, 1983. \MR{703736
  (84j:20057)}

\bibitem[Pic00]{PicantinPhD}
Matthieu Picantin.
\newblock {\em Petits groupes gaussiens}.
\newblock PhD thesis, Universit\'e de Caen, 2000.

\bibitem[Pic01]{Picantin01}
Matthieu Picantin.
\newblock The center of thin {G}aussian groups.
\newblock {\em J. Algebra}, 245(1):92--122, 2001. \MR{1868185 (2002h:20053)}

\bibitem[RS55]{RedeiSzep55}
L.~R{\'e}dei and J.~Sz{\'e}p.
\newblock Die {V}erallgemeinerung der {T}heorie des {G}ruppenproduktes von
  {Z}appa-{C}asadio.
\newblock {\em Acta. Sci. Math. Szeged}, 16:165--170, 1955. \MR{0075941
  (17,823d)}

\bibitem[Sz{\'e}50]{Szep50}
J.~Sz{\'e}p.
\newblock On the structure of groups which can be represented as the product of
  two subgroups.
\newblock {\em Acta Sci. Math. Szeged}, 12(Leopoldo Fejer et Frederico Riesz
  LXX annos natis dedicatus, Pars A):57--61, 1950. \MR{0037296 (12,239e)}

\bibitem[Sz{\'e}51]{Szep51}
J.~Sz{\'e}p.
\newblock Zur {T}heorie der endlichen einfachen {G}ruppen.
\newblock {\em Acta Sci. Math. Szeged}, 14:111--112, 1951. \MR{0048439
  (14,13j)}

\bibitem[Sz{\'e}62]{Szep62}
Jen{\"o} Sz{\'e}p.
\newblock Sulle strutture fattorizzabili.
\newblock {\em Atti Accad. Naz. Lincei Rend. Cl. Sci. Fis. Mat. Nat. (8)},
  32:649--652, 1962. \MR{0148753 (26 \#6259)}

\bibitem[Tak81]{Takeuchi81}
Mitsuhiro Takeuchi.
\newblock Matched pairs of groups and bismash products of {H}opf algebras.
\newblock {\em Comm. Algebra}, 9(8):841--882, 1981. \MR{611561 (83f:16013)}

\bibitem[Zap42]{Zappa42}
Guido Zappa.
\newblock Sulla costruzione dei gruppi prodotto di due dati sottogruppi
  permutabili tra loro.
\newblock In {\em Atti {S}econdo {C}ongresso {U}n. {M}at. {I}tal., {B}ologna,
  1940}, pages 119--125. Edizioni Cremonense, Rome, 1942. \MR{0019090 (8,367d)}

\end{thebibliography}

\bigskip
\noindent
\begin{minipage}[t]{0.52\textwidth}
\noindent\textbf{Volker Gebhardt}\\
\noindent \texttt{v.gebhardt@westernsydney.edu.au}
\end{minipage}
\hfill
\begin{minipage}[t]{0.47\textwidth}
\noindent\textbf{Stephen Tawn}\\
\noindent \texttt{stephen@tawn.co.uk}\\
\noindent \url{http://www.stephentawn.info}
\end{minipage}
\medskip
\begin{center}
Western Sydney University \\
Centre for Research in Mathematics\\
Locked Bag 1797, Penrith NSW 2751, Australia\\
\noindent URL: \url{http://www.westernsydney.edu.au/crm}
\end{center}

\end{document}